\documentclass[a4paper,11pt,leqno]{article}   
\usepackage{amssymb, amsmath, amsthm, graphicx, color, epsfig, amsfonts, tikz, url} 
\usepackage{xtab,booktabs}
\usepackage{enumitem}
\usepackage[normalem]{ulem}

\usepackage{thmtools} 
\newtheorem{thm}{Theorem}[section]
\newtheorem{lem}[thm]{Lemma}
\newtheorem{cor}[thm]{Corollary}
\newtheorem{prop}[thm]{Proposition}
\newtheorem{assn}[thm]{Assumption}

\declaretheoremstyle[bodyfont=\normalfont,qed=$\square$]{remark}
\declaretheorem[style=remark,numberlike=thm,name=Remark]{rem}
\numberwithin{equation}{section}

\def\R{{\mathbb{R}}}

\def\Z{{\mathbb{Z}}}
\def\T{{\mathbb{T}}}

\def\E{{\mathbb E}}
\def\e{\varepsilon}
\renewcommand{\P}{{\mathbb P}}

\allowdisplaybreaks

\newcommand{\bzero}{\mathbf 0}

\newcommand{\ba}{\mathbf a}

\newcommand{\bF}{\mathbf F}

\newcommand{\bi}{\mathbf i}
\newcommand{\bI}{\mathbf I}

\newcommand{\bJ}{\mathbf J}
\newcommand{\bk}{\mathbf k}
\newcommand{\bK}{\mathbf K}

\newcommand{\blam}{{\boldsymbol \lambda}}
\newcommand{\bLam}{{\boldsymbol \Lambda}}

\newcommand{\bP}{\mathbf P}
\newcommand{\bp}{\mathbf p}
\newcommand{\bphi}{{\boldsymbol \varphi}}

\newcommand{\bq}{\mathbf q}

\newcommand{\bv}{\mathbf v}

\newcommand{\cD}{\mathcal D}

\newcommand{\cK}{\mathcal K}

\newcommand{\cP}{\mathcal P}
\newcommand{\cQ}{\mathcal Q}

\newcommand{\cU}{\mathcal U}
\newcommand{\cV}{\mathcal V}

\newcommand{\lrang}[1]{{\left\langle #1 \right\rangle }}

\newcommand{\law}{\operatorname{law}}
\newcommand{\supp}{\operatorname{supp}}

\renewcommand{\tilde}{\widetilde}

\definecolor{darkgreen}{rgb}{0,0.7,0}
\definecolor{orange}{rgb}{1,0.45,0}

\long\def\comm#1{}

\title{Incompressible limit for weakly asymmetric simple exclusion processes coupled through collision}
\author{Patrick van Meurs$\,^{1)}$, Kenkichi Tsunoda$\,^{2)}$ and Lu Xu$\,^{3)}$}
\date{}

\begin{document}
\maketitle
 
%
%
%


\begin{abstract}
We establish the incompressible limit of weakly asymmetric simple exclusion processes coupled through particle collisions. The incompressible limit depends on various parameters in the particle system and is linked to fluid dynamics equations. Our main contributions to previous results are the extension of the parameter space and the focus on local particle jumps. Our proof uses the relative entropy method. The main novelties in the proof are a Boltzmann--Gibbs principle (a replacement lemma) and a spectral gap estimate. 

\footnote{
\hskip -6mm 
${}^{1)}$ Faculty of Mathematics and Physics, Kanazawa University, Kakuma, Kanazawa 920-1192, Japan.
e-mail: pjpvmeurs@staff.kanazawa-u.ac.jp \\
${}^{2)}$ Faculty of Mathematics, Kyushu University, 744 Motooka, Nishi-ku, Fukuoka, 819-0395, Japan. e-mail: tsunoda@math.kyushu-u.ac.jp \\
${}^{3)}$ Gran Sasso Science Institute, Viale Francesco Crispi 7, 67100 L'Aquila, AQ, Italy. e-mail: lu.xu@gssi.it}
\footnote{
\hskip -6mm
Abbreviated title $($running head$)$: Incompressible limit of particle systems with collisions.}
\footnote{
\hskip -6mm
2020MSC: 60K35, 82C22.
}
\end{abstract}

\tableofcontents

\section{Introduction}
\label{s:intro}

A major open problem in nonequilibrium statistical mechanics is the derivation of the incompressible Navier--Stokes equation from microscopic Hamiltonian dynamics. The main difficulty is the poor knowledge of the ergodic properties of these dynamics. To side-step this difficulty, a common approach is to replace the deterministic Hamiltonian dynamics by stochastic lattice gases. 

Esposito, Marra and Yau \cite{EMY} used this approach to derive the incompressible Navier--Stokes equation as a limit of stochastic lattice gases. The state space of this lattice gas is the discrete torus $\T_N^d = \{1, \ldots, N\}^d$, where $d \ge 3$ is the dimension and $N^{-1}$ is the macroscopic distance between neighboring lattice sites. On $\T_N^d$ particles of several species evolve. Each species is labeled with a velocity $v \in \cV$ for a finite velocity set $\cV \subset \R^d$. The dynamics consists of two parts: 
asymmetric simple exclusions  among particles of the same species, and binary collisions between particles of different species. The asymmetry is such that the particles of species $v$ cause a macroscopic drift given by $v$, and the collisions are such that the momentum of the colliding particles is conserved. The \textit{incompressible limit} is then proved when time and the velocity field are suitably rescaled.
Namely, when the initial state is given by adding a small perturbation of size $N^{-1}$ on the global equilibrium, the corresponding perturbation field satisfies the incompressible limit equation. For $d=3$, a special choice of $\cV$ and under assumptions on the perturbation field, this equation becomes the incompressible Navier--Stokes equation.
\comm{[TASEP: movement ONLY in 1 direction. ASEP: asym part is of the same order as the sym part. \textit{Simple} = nearest-neighbor.]}

This result has set the stage for many further studies that deepen the connection between stochastic lattice gas models and fluid mechanics equations. 
Quastel and Yau \cite{QY} add to the results in \cite{EMY} by building weak solutions to the incompressible Navier--Stokes equation and by providing convergence rates through establishing large deviation principles.
Beltran and Landim \cite{BL} consider a slightly different model in which particles perform \textit{mesoscopic} range jumps with \textit{weak} asymmetry of size $N^{-1+a}$ with parameter $a\in(0,1)$. The incompressible limit is proved under the diffusive time scale for all dimensions $d\ge1$. While this setting allows for a substantially simpler proof based on the relative entropy method, it requires in addition a spectral gap estimate for the full dynamics of the system, which is established only for the specific velocity set $\cV=\{\pm e_j;j=1,\ldots,d\}$.
A similar proof strategy has been exploited to derive both viscous and inviscid Burgers equation from equilibrium perturbations of microscopic conservation laws, see e.g.\ \cite{EMY94,JLT,Seppa,TV,XZ}. 
\comm{[\emph{fluid mechanics equations} cover NS and Burgers equation]}

In the present paper, we aim to contribute to these works. Our setting is most closely related to that in \cite{BL}: we also consider jump probabilities in which the asymmetry is of size $O(N^{-1+a})$ for some $a > 0$ small enough.
However, there are two main differences:
\begin{enumerate}
  \item We consider a large class of sets $\cV$ (see Proposition \ref{p:cV} and Theorem \ref{t} below). In particular, we do not impose the common assumption in \cite{BL,EMY,QY} that the velocity set $\cV$ is closed under reflections and permutations. 
  \item We do not force mesoscopic jumps, and consider instead local jumps (see \eqref{pN} below). This is a more standard modeling choice; see e.g.\ \cite{KL}.
\end{enumerate}
We consider a large class of sets $\cV$ to obtain a general form of the incompressible limit, which is a system of equations (see \eqref{HDL:varphi} below). With this large class we can obtain various couplings between these equations. We consider this as an important generalization of the simple choice $\cV=\{\pm e_j;j=1,\ldots,d\}$ for which most of the equations are decoupled. 

Considering local jumps instead of mesoscopic jumps is more challenging. Indeed, in \cite{BL} the mesoscopic range of the jumps is used to establish the so-called one-block estimate (see the conditions in (2.13) of \cite{BL}). For local jumps this difficulty has recently been overcome in the case of the weakly asymmetric exclusion process in \cite{JLT}, which is a simplified version of our particle system with only one velocity and no collisions. In \cite{JLT} the authors invoke a sharp moment estimate, whose origin is in \cite{JM2}, and obtain the incompressible limit. Unfortunately, this moment estimate can not be applied to our setting directly, due to the mixing caused by the collision dynamics. 

The main result of this paper is the incompressible limit in Theorem \ref{t}. The proof follows Yau's relative entropy method as in \cite{BL}. The main step is the construction of a Boltzmann-Gibbs principle (a replacement lemma). As mentioned above, we cannot follow the proof in \cite{BL} due to the local jumps. Instead, we follow the proof method developed in \cite{EFHPS,FvMST}, where a similar Boltzmann-Gibbs principle is proven for Glauber--Zero-Range and Glauber--Kawasaki processes. A key step in this proof is a spectral gap estimate on the generator of the particle system. To obtain it, we reduce the problem to the spectral gap of a system with pure collision dynamics and extend the proof method used in \cite[Section 6]{BL} to a wider class of velocity sets $\cV$.

We believe that the incompressible limit in Theorem \ref{t} holds for a larger class of velocity sets $\cV$, namely those for which the total mass and momentum are the only conserved quantities of the dynamics. This would include e.g.\ Model II in \cite{EMY} for which the incompressible Navier--Stokes equation was obtained.
However, the spectral gap estimate remains an open problem for general $\cV$, mainly due to the complicated structure of the collision set.
Another open problem concerns the dynamics with boundary effects.
For boundary-driven, multi-species weakly asymmetric exclusion with collisions, the hydrodynamic limit is proved in \cite{AGS,Sim}, while the corresponding incompressible limit is yet unclear.

The paper is organized as follows. In Section \ref{s:setting} we introduce the setting, state our main results (i.e.\ Theorem \ref{t} on the incompressible limit and Proposition \ref{p:cV} on the spectral gap estimate) and discuss our findings. The remaining sections contain the proofs of Theorem \ref{t} and  Proposition \ref{p:cV}; in Section \ref{s:ent:prod} we compute the entropy production, in Section \ref{s:BG} we state and prove the Boltzmann-Gibbs principle, in Section \ref{s:SG} we prove Proposition \ref{p:cV}, and in Appendix \ref{a:props:mulamN} we perform relatively standard computations.

\section{Setting and main result}
\label{s:setting}

As above, let $\T_N^d = \{1, \ldots, N\}^d$ be the $d$-dimensional discrete torus of size $N$, where $N, d \geq 1$ are integers. We denote sites on $\T_N^d$ by $x,y,z \in \T_N^d$.  We consider particles of various species indexed by velocities $v$ from a fixed, finite velocity set $\cV \subset \R^d$. We impose assumptions on $\cV$ whenever required, and postpone the list of standing assumptions on $\cV$ to Section \ref{s:sett:ass}. The particle configuration space is $X_N := (\{0,1\}^\cV)^{\T_N^d}$. For a configuration $\eta \in X_N$, we denote by $\eta_x(v) \in  \{0,1\}$ the absence or presence of a particle at $x$ of species $v$. We interpret $\eta_x \in \{0,1\}^\cV$ as the list of occupancies at site $x$ for each particle species. \comm{[For PvM to remember: a set of ergodic states means that the Markov chain is irreducible on these states. A micro-canonical surface is a subset of states in $X_N$ consisting of \textbf{all} states $\eta$ for which $\bI(\eta) = \bi$ for some fixed $\bi$.]}

On $X_N$ we consider the particle system generated by
\begin{equation} \label{LN}
  L_N = N^2( L_N^{ex} + L_N^c ),
\end{equation}
in which the factor $N^2$ represents the diffusive rescaling in time.
The generator $L_N^{ex}$ describes a unit time nearest-neighbor exclusion process 
given by \comm{[Usually, "simple" exclusion process means nearest-neighbor, but this terminology is not used so carefully sometimes.]}
\begin{equation*}
  (L_N^{ex} f)(\eta) 
  = \sum_{v \in \cV} \sum_{ \substack{ x, y\in\T_N^d \\ |x-y|=1 }} \eta_x(v) \big( 1- \eta_y(v) \big)
p_N(y-x,v) \left\{  f\left(  \eta^{x,y,v}\right)  -f\left(\eta\right)  \right\}.
\end{equation*} 
Here, $f : X_N \to \R$ is a test function and
\[
  \eta_z^{x,y,v}(w) = \begin{cases}
    \eta_y(w)
    &\text{if } w = v \text{ and } z = x,  \\
    \eta_x(w)
    &\text{if } w = v \text{ and } z = y,  \\
    \eta_z(w)
    &\text{otherwise.} 
  \end{cases}
\]
denotes the particle configuration obtained from $\eta$ by swapping the occupancies of species $v$ at sites $x$ and $y$. The transition probabilities  $p_N(z,v)$ for $|z| = 1$ are given by
\begin{equation} \label{pN}
  p_N (z,v) := 1 + \frac{z \cdot v}{N^{1-a}}, 
\end{equation}
where $a \in (0, 1)$ is a given parameter. In our main Theorem \ref{t} we will impose a further explicit upper bound on $a$. Note that under $L_N^{ex}$ each particle species evolves independently. We focus on nearest-neighbor jumps for simplicity; in Remark \ref{r:t:genl:jumps} we mention the extension to local jumps. \comm{the corresponding constant $a_{BL}$ in \cite{BL} is related to $a$ as follows; $a_{BL} = 1-a$ and thus $a = 1-a_{BL}$} 

The second generator in \eqref{LN} describes particle collisions. It is the same as in the previous studies cited in Section \ref{s:intro}. It is given by
\begin{equation}\label{LNc}
  (L_N^c f)(\eta) 
  = \sum_{ x \in \T_N^d } \sum_{q \in \cQ} 
p(x,q,\eta) \left\{  f\left(  \eta^{x,q}\right)  -f\left(\eta\right)  \right\}.
\end{equation}
Here, $f : X_N \to \R$ is again a test function and
\[
  \cQ := \{ q = (v,w,v',w') \in \cV^4 : v + w = v' + w' \}
\]
is the set of all possible momentum-conserving collisions between particles of different species, where $v$ and $w$ are the velocities of the colliding particles, and  $v'$ and $w'$ the velocities of the particles after the collision. In \eqref{LNc}, $p$ is the indicator function on configurations $\eta$ at which a collision $q = (v,w,v',w') \in \cQ$ can happen at $x$, i.e.\ 
\[
  p(x,q,\eta) := \eta_x(v) \eta_x(w) (1 - \eta_x(v')) (1 - \eta_x(w')).
\]
Finally, $\eta^{x,q}$ denotes the configuration obtained from $\eta$ after a collision $q$ happened at $x$. For $q = (v_0,v_1,v_2,v_3) \in \cQ$ it is given by
\[
  \eta_y^{x,q}(w) := \begin{cases}
    \eta_y(v_{j+2})
    &\text{if } y = x \text{ and } w = v_j \text{ for some } 0 \leq j \leq 3, \\
    \eta_y(w)
    &\text{otherwise,} 
  \end{cases}
\] 
where the index of $v_{j+2}$ should be understood modulo $4$. 

In what follows, we denote by $\{ \eta^N(t) \}_{t \geq 0}$ the process on $X_N$ generated by $L_N$ and by $\mu_t^N$ the law of $\eta^N(t)$. We interpret $\mu^N := \mu_0^N$ as the initial distribution. Finally, we set $\P_{\mu^N}(\cdot)$ as the distribution of $t \mapsto \eta^N(t)$ on the RCLL path space $D([0,\infty), X_N)$ conditioned on $\law \eta^N(0) = \mu^N$, and $\E_{\mu^N}$ as the corresponding expectation.

\subsection{Conserved quantities}

Given $\eta \in X_N$ and a site $x$, the mass and $k$-th component of the momentum for the particles at $x$ are respectively
\begin{equation*}
  I_0 (\eta_x) := \sum_{v \in \cV} \eta_x(v),
  \qquad I_k (\eta_x) := e_k \cdot \sum_{v \in \cV} v \eta_x(v) = \sum_{v \in \cV} v_k \eta_x(v).
\end{equation*}
To unify notation, we set
\begin{equation*}
  \bI(\eta_x) := (I_0(\eta_x), \ldots, I_d(\eta_x)) \in \R^{d+1}.
\end{equation*}
We refer to $\bI(\eta_x)$ as the local mass-momentum values of $\eta$ at $x$.
Moreover, given $v \in \cV$, we set
\begin{equation*}
  v_0 := 1, \qquad \bv := (v_0, v_1, \ldots, v_d) \in \R^{d+1}.
\end{equation*}
Then 
\begin{equation} \label{bI}
\bI(\eta_x) = \sum_{v \in \cV} \bv \eta_x(v). 
\end{equation}

The $d+1$ conserved quantities under $L_N$ are the (total) mass-momentum values
\begin{equation} \label{csv:quan}
  \sum_{x \in \T_N^d} \bI(\eta_x) = \sum_{v \in \cV} \bv \sum_{x\in\T_N^d} \eta_x(v).
\end{equation}
Indeed, $L_N^{ex}$ conserves the number of particles for each species $v \in \cV$, and $L_N^c$ conserves $\bI(\eta_x)$ for each site $x$. The latter can be rephrased as Lemma \ref{l:LNc:0}.

\begin{lem} \label{l:LNc:0}
If $f(\eta)$ only depends on $\eta$ through the mass and momentum at $x$, i.e.\ $f(\eta) = \tilde f ( ( \bI(\eta_x))_{x \in \T_N^d} )$, then $L_N^c f = 0$.
\end{lem}


In general, the conserved quantities in \eqref{csv:quan} need not be unique. What we mean by uniqueness is that all micro-canonical ensembles of $X_N$ are ergodic under $L_N$. These ensembles are given by 
\begin{equation*}
  \Big\{ \eta \in X_N : \frac1{N^d} \sum_{x \in \T_N^d} \bI(\eta_x) =\bp \Big\},
\end{equation*}
where $\bp \in \R^{d+1}$ denotes attainable averaged mass-momentum values.
In our main Theorem \ref{t} we will restrict the choice of $\cV$ such that all micro-canonical ensembles are ergodic. 

\subsection{Invariant measures}  
\label{s:sett:invmeas}

We introduce the invariant measures $\mu_\blam^N$ of $L_N$, which are parametrized by the chemical potential $\blam = (\lambda_0, \ldots, \lambda_d) \in \R^{d+1}$. They are given by
\begin{equation} \label{mulam:prod}
  \mu_\blam^N (\eta)
  = \prod_{x \in \T_N^d} \prod_{v \in \cV} \frac{ e^{\blam \cdot \bv \eta_x(v)} }{e^{ \blam \cdot \bv } + 1},
\end{equation}
i.e.\ they are the product Bernoulli measure on $\{ 0,1 \}^{\T_N^d \times \cV}$ with parameter $\theta(\blam \cdot \bv)$, where
\begin{equation} \label{theta}
  \theta : \R \to (0,1), \qquad \theta(\alpha) := \frac{e^\alpha}{e^\alpha + 1}.
\end{equation}

Next we list several properties of $\mu_\blam^N$. Note that $\mu_\blam^N (\eta)$ can alternatively be expressed as
\begin{equation} \label{mulam:bI}
  \mu_\blam^N (\eta)
  = \prod_{x \in \T_N^d} \frac{ e^{\blam \cdot \sum_{v \in \cV} \bv \eta_x(v)} }{\prod_{v \in \cV} (e^{ \blam \cdot \bv } + 1)}
  = \frac1{\prod_{v \in \cV} (e^{ \blam \cdot \bv } + 1)} \prod_{x \in \T_N^d} e^{\blam \cdot \bI(\eta_x)}.
\end{equation}
Then, by Lemma \ref{l:LNc:0} we obtain $L_N^c \mu_\blam^N = 0$. The remaining properties that follow are proved in Appendix \ref{a:props:mulamN}. $\mu_\blam^N$ is invariant both under $L_N^{ex}$ and $L_N^c$, and therefore also under $L_N$. $L_N^c$ is symmetric with respect to $\mu_\blam^N$. The symmetric and anti-symmetric part of $L_N^{ex}$ with respect to $\mu_\blam^N$ are respectively
  \begin{equation}  \label{LNexs}
    (L_N^{ex,s} f)(\eta) 
  = \sum_{v \in \cV} \sum_{ \substack{ x, y\in\T_N^d \\ |x-y|=1 }} \eta_x(v) \big( 1- \eta_y(v) \big)
\left\{  f\left(  \eta^{x,y,v}\right)  -f\left(\eta\right)  \right\}.
  \end{equation}
  and
  \begin{equation} \label{LNexa} 
    (L_N^{ex,a} f)(\eta) 
  = \frac1{N^{1-a}} \sum_{v \in \cV} \sum_{ \substack{ x, y\in\T_N^d \\ |x-y|=1 }} ((y-x) \cdot v) \eta_x(v) \big( 1- \eta_y(v) \big)
\left\{  f\left(  \eta^{x,y,v}\right)  -f\left(\eta\right)  \right\}.
  \end{equation} 

For later use, we re-parametrize $\mu_\blam^N$ in terms of the averaged mass-momentum values $\bp$. Given $\blam$, these values are given by
\begin{equation} \label{bP}
  \bP : \R^{d+1} \to \R^{d+1}, \qquad \bP(\blam) := E_{\mu_{\blam}^N} [\bI(\eta_x)] = \sum_{v \in \cV} \bv \theta(\blam \cdot \bv).
\end{equation} 
We note that for $\blam = \bzero$ we have 
\begin{equation} \label{bpstar}
  \bP(\bzero) 
  =  \frac12 \sum_{v \in \cV} \bv
  =: \bp_*.
\end{equation} 
To re-parametrize $\mu_\blam^N$, we need to invert $\bP$. With this aim, we set $\cU \subset \R^{d+1}$ as the range of $\bP$. For the moment, we assume that $\cV$ is such that $\bP$ is a diffeomorphism onto $\cU$ (in our main Theorem \ref{t} a much weaker and explicit assumption on $\cV$ will be sufficient). We set 
$\bLam := \bP^{-1} : \cU \to \R^{d+1}$, and re-parametrize the invariant measures as
\[
  \nu_\bp^N := \mu_{\bLam(\bp)}^N.
\]

\subsection{Informal derivation of the incompressible limit}
 
From the hydrodynamic limits obtained in \cite{AGS,Sim} we expect the hydrodynamic limit of our process $\eta^N(t)$ to be \comm{[For PvM to remember: HDL refers to the limit to the $\bp$ variable, and incompr lim refers to the limit to the $\bphi$ variable]}
\begin{equation} \label{HDL:rhop}
  \partial_t \bp + 2 N^a \sum_{v \in \cV} \big( v \cdot \nabla \big[ \chi \big(\theta( \bv \cdot \bLam(\bp) ) \big) \big] \big) \bv  = \Delta \bp + O(N^{-1+a}),
\end{equation}
where $\chi(\vartheta) := \vartheta(1-\vartheta)$ is the compressibility and $p_k(t,\cdot) \in C^2(\T^2)$ is the supposed deterministic limit of 
\[
   \frac1{N^d} \sum_{x\in\T_N^d} I_k(\eta_x^N(t)) \delta_{x/N} \qquad \text{for } k = 0,\ldots,d,
\]
which is the random empirical measure on $\T^d$ of the $k$-th conserved quantity of the process $\eta^N(t)$ at time $t$.

As is clear from the prefactor $N^a$ in \eqref{HDL:rhop}, the hydrodynamic scaling only fits for the usual weakly asymmetric interactions (i.e.\ $a = 0$) whereas we have $a \in (0,1)$ in \eqref{pN}. Hence, we apply the incompressible scaling as in \cite{BL,EMY,QY}. In preparation for this we expand $\bp$ around $\bp_*$ (recall \eqref{bpstar}) as 
\begin{equation} \label{varphik:def}
  \bphi := N^a (\bp - \bp_*), \qquad \bp = \bp_* + \frac{\bphi}{N^a}. 
\end{equation}
Substituting this expression in \eqref{HDL:rhop} and neglecting terms of size $O(N^{-a})$, we obtain from a computation the desired equation of $\bphi$; see \eqref{HDL:varphi} below. We omit this computation here. Instead, the proof of our main Theorem \ref{t} demonstrates that \eqref{HDL:varphi} is indeed the incompressible limit. It is given by
\begin{equation} \label{HDL:varphi} 
  \partial_t \bphi = \sum_{v \in \cV} (\bv \cdot A_\cV \bphi) (\bv \cdot A_\cV \nabla \bphi v) \bv + \Delta \bphi,
\end{equation}
where 
\begin{equation} \label{AV}
  A_\cV := 2 \Big[\sum_{v \in \cV} \bv \otimes \bv\Big]^{-1} \in \R^{(d+1) \times (d+1)}
\end{equation}
is a symmetric matrix and $\nabla \bphi$ is the Jacobian, whose entries are given by $[\nabla \bphi]_{kl} = \partial_l \varphi_k$ with $0 \leq k \leq d$ and $1 \leq l \leq d$. The invertibility of \eqref{AV} follows from the invertibility of $\bP$.
In components, \eqref{HDL:varphi} reads \comm{[detailed computations in Patrick's written notes p.75]}
\[
  \partial_t \varphi_k = \sum_{i,j=0}^d \sum_{l=1}^d C_{kijl} \varphi_i \partial_l \varphi_j + \Delta \varphi_k \qquad (0 \leq k \leq d)
\] 
with constants
\[
  C_{kijl} = C_{kijl}(\cV) := \sum_{v \in \cV} v_k (A_\cV \bv)_i (A_\cV \bv)_j v_l.
\]

\subsection{Assumptions}
\label{s:sett:ass}

Here we state the precise list of assumptions that we need for our main Theorem \ref{t}.
First, for the incompressible limit \eqref{HDL:varphi} to be well defined, we need the following on $\cV$:

\begin{assn} \label{a:AV}
Recall that $\bv=(1,v)$ for $v\in\cV$.
The matrix $\sum_{v\in\cV} \bv \otimes \bv$ is invertible.
\end{assn} 

Second, we comment below in Remark \ref{r:bphi} on the existence of a solution $\bphi$ of \eqref{HDL:varphi}; Theorem \ref{t} only applies on a time interval where such a solution exists.

\begin{rem} \label{r:bphi}
For $T > 0$ small enough, the existence and uniqueness of a classical solution of \eqref{HDL:varphi} for any sufficiently smooth initial condition follows from parabolic regularity theory; see e.g.\ \cite{Ang}, which is based on the theory of DaPrato and Grisvard of 1979.
For general $T > 0$ we are not aware of existence results in the literature. 
\end{rem}

Third, we need to compare a solution $\bphi$ of \eqref{HDL:varphi} to the distribution $\mu_t^N \in \cP(X_N)$ of the process $\eta^N(t)$. We do this by constructing a measure $\nu_t^N \in \cP(X_N)$ from $\bphi$. For this construction, we first take the field $\bp$ corresponding to $\bphi$ from \eqref{varphik:def}. To obtain the corresponding field $\blam$, we rely on the following proposition.

\begin{prop} \label{p:nutN}
Let $\cV$ be as in Assumption \ref{a:AV}. Then, $\bP$ is invertible in a neighborhood of $\bp_* = \bP(\bzero)$ and the corresponding inverse map $\bLam$ is smooth. Moreover, for a classical solution $\bphi$ of \eqref{HDL:varphi} on $[0,T]$ and for $N$ large enough, the function
\[
  \blam (t,x)
  := \bLam \Big(\bp_* + N^{-a} \bphi \Big(t, \frac xN \Big) \Big)
\]
is well-defined for all $t \in [0,T]$ and all $x \in \T_N^d$.
\end{prop}

\begin{proof}
It is obvious from \eqref{bP} that $\bP$ is smooth. A computation (see Section \ref{s:ent:prod:rel}) shows that $\nabla \bP(\bzero) = \frac14 \sum_{v \in \cV} \bv \otimes \bv$, which by Assumption \ref{a:AV} is invertible. Then, by the Implicit Function Theorem $\bP$ is invertible in a neighborhood of $\bP(\bzero) = \bp_*$ on which $\bLam = \bP^{-1}$ is smooth. Since $\| \bphi \|_\infty < \infty$, the last statement of Proposition \ref{p:nutN} follows.
\end{proof}

With the function $\blam$ constructed from a solution $\bphi$ in Proposition \ref{p:nutN}, we set (in accordance to \eqref{mulam:prod})
\begin{equation} \label{nutN}
  \nu_t^N 
  := \mu_{\blam(t,\cdot)}^N 
  := \prod_{x \in \T_N^d} \prod_{v \in \cV} \frac{ e^{\blam(t,x) \cdot \bv \eta_x(v)} }{e^{ \blam(t,x) \cdot \bv } + 1}
\end{equation} 
as the product measure associated to $\blam(t,\cdot)$.

Fourth, we need that $\cV$ is such that the total mass-momentum values in \eqref{csv:quan} are the only conserved quantities. In fact, our proof of Theorem \ref{t} requires a stronger, technical assumption, namely that a certain localized version of $L_N$ satisfies a sufficient spectral gap estimate. To state this estimate, we introduce some notation.
Let $M \ge 1$ be an integer and
\begin{equation*}
  \Lambda_M:=\big\{x\in\mathbb Z^d;|x_i| \le M,\forall\,i=1,\ldots,d\big\}
\end{equation*}
be a box in $\Z^d$.  
Observe that $|\Lambda_M|=(2M+1)^d$ and that $\Lambda_M$ can be embedded into $\T_N^d$ if $N \ge 2M+1$.
For $x\in\T_N^d$ and $\eta \in X_N$, let
\begin{equation*} 
  \bI_x^M(\eta)
  :=\frac1{|\Lambda_M|}\sum_{z \in x+\Lambda_M} \bI(\eta_z)
\end{equation*}
be the average over the shifted box $x + \Lambda_M$ of the local mass-momentum values.
Denote by $\cD_M$ the set of all possible values of $\bI_0^M$ when $\eta$ runs over $\{0,1\}^{\Lambda_M\times\cV}$.
Define the micro-canonical surface
\begin{equation}\label{mic-can space}
  Y_M(\bi):=\big\{\eta\in\{0,1\}^{\Lambda_M\times\cV} \,|\, \bI_0^M(\eta)=\bi\big\}, \quad \forall\,\bi\in\cD_M.
\end{equation}
Let $\nu_{M,\bi}$ be the micro-canonical Gibbs measure on $Y_M(\bi)$ given by
\begin{equation}\label{mic-can measure}
  \nu_{M,\bi}(\eta):=\nu_\bp^N\big(\eta\,|\,Y_M(\bi)\big) = \big|Y_M(\bi)\big|^{-1}, \quad \forall\,\eta \in Y_M(\bi).
\end{equation}
The second equality follows from the fact that the invariant measure $\nu_\bp^N$ is invariant under both particle jumps and collisions.
Thus, $\nu_{M,\bi}$ is the uniform measure on $Y_M(\bi)$. In particular, it is independent of $\bp$.

Recall the generators defined in \eqref{LNexs} and \eqref{LNc}.
Their restrictions to $\Lambda_M$ without specific boundary condition are respectively given by
\begin{equation}\label{restricted operators}
  \begin{aligned}
    L_M^{ex,s}f(\eta) &= \sum_{v\in\cV} \sum_{\substack{x,y\in\Lambda_M\\|x-y|=1}} \eta_x(v) \big(1-\eta_y(v)\big) \big[f(\eta^{x,y,v})-f(\eta)\big],\\
    L_M^{c}f(\eta) &= \sum_{q\in\cQ} \sum_{x\in\Lambda_M} p(x,q,\eta) \big[f(\eta^{x,q})-f(\eta)\big].
  \end{aligned}
\end{equation}
Note that both are independent of $N$.

\begin{assn}[Spectral gap]\label{a:sg}
There exist constants $\kappa$ and $C$, both depending only on $d$ and $\cV$, such that for all $M \ge 1$, $\bi\in\cD_M$ and $f:Y_M(\bi) \to \R$, 
\begin{equation*}
  E_{\nu_{M,\bi}} \left[ \big(f - E_{\nu_{M,\bi}}[f]\big)^2 \right] \le CM^\kappa \big\langle f, -(L_M^{ex,s}+L_M^{c})f \big\rangle_{\nu_{M,\bi}}.
\end{equation*}
\end{assn}

Since Assumptions \ref{a:AV} and \ref{a:sg} on $\cV$ are implicit, we provide in Proposition \ref{p:cV} a class of $\cV$ for which both assumptions hold. \comm{[For PvM not to forget: The bdy condns in Assumption \ref{a:sg} are not important; it also applies to $L_N$ on $\T_N^d$. Then, it gives exponential in-time convergence of $\mu_t^N$ to the micro-canonical Gibbs measure, irrespective of the initial condition. This show that $Y_N(\bp)$ is irreducible, i.e.\ we have uniqueness of csvd quans.]}

\begin{prop}[Possible class of $\cV$]\label{p:cV}
Let $n \geq 1$, $v_*, v_1, v_2, \ldots, v_n \in \R^d$ and
\begin{equation}\label{eq:cV}
  \cV=\{v_* \pm v_\ell;\ell=1,\ldots,n\}.
\end{equation}
Then the following two statements hold:
\begin{enumerate}[label=(\roman*)]
  \item \label{p:cV:AV} if $\{v_1, \ldots, v_n\}$ spans $\R^d$, then $\cV$ satisfies Assumption \ref{a:AV};
  \item \label{p:cV:sg} if the only integer-valued solution to 
\[\alpha_1v_1+\cdots+\alpha_nv_n=0\] 
is $\alpha_1=\cdots=\alpha_n=0$, then $\cV$ satisfies Assumption \ref{a:sg} with 
\begin{equation*} 
  \kappa = (2 n + 1)d + 2.
\end{equation*}
\end{enumerate} 
\end{prop}
 
We prove Proposition \ref{p:cV} in Section \ref{s:SG}. Here, we give several comments to parse it. First, $v_*$ is the common drift direction. Second, the `discrete independence' condition in \ref{p:cV:sg} implies in particular that only particles with opposite velocity can collide, i.e.
\begin{equation*} 
  \cQ = \{ (v,-v,w,-w) : v,w \in \cV \}.
\end{equation*}
Third, the conditions in \ref{p:cV:AV} (when $n \geq d$) and \ref{p:cV:sg} only exclude special cases; indeed, if $v_1, \ldots, v_n$ are chosen independently from an absolutely continuous distribution, then both conditions hold almost surely. 
Fourth, we do not expect the value of $\kappa$ to be optimal.

\subsection{Main result}
\label{s:sett:t}

To state our main Theorem \ref{t}, we recall that  $d \geq 1$, that $\mu_t^N$ is the law of the process $\eta^N(t)$ generated by $L_N$ (see \eqref{LN}) with given initial distribution $\mu^N$, and that for a solution $\bphi$ of \eqref{HDL:varphi} we set for $N$ large enough $\bp = \bp_* + N^{-a}\bphi$ as the corresponding field of the conserved quantities and $\nu_t^N$ as the measures associated to $\blam(t, \cdot) = \bLam (\bp(t,\cdot))$ (see \eqref{nutN} and Proposition \ref{p:nutN}). In addition, we introduce the relative entropy for measures $\mu$ and $\nu$ on $X_N$ with $\nu$ having full support as
\begin{equation*}
  H(\mu | \nu)
  := \sum_{\eta \in X_N} \mu (\eta) \log \frac{\mu (\eta)}{\nu (\eta)}.
\end{equation*}
Finally, we introduce $C^{k,\ell}([0,T] \times \T^d; \R^{d+1})$ as the space of functions that are $k$ times continuously differentiable in time and $\ell$ times continuously differentiable in space.

\begin{thm}[Main] \label{t}
Let $\cV$ be such that Assumptions \ref{a:AV} and \ref{a:sg} hold for some constant $\kappa > 0$ and let 
\[
  0 < a < \frac d{\kappa + 2d}.
\] 
Let $\bphi \in C^{1,3}([0,T] \times \T^d; \R^{d+1})$ be a classical solution of \eqref{HDL:varphi} for some $T > 0$ (see Remark \ref{r:bphi} above) and let $\nu_t^N$ be the corresponding measure mentioned above for all $t \in [0,T]$ and all $N$ large enough. Let $\mu^N$ be an initial distribution and let $\mu_t^N$ be the law of the process $\eta^N(t)$ generated from it by $L_N$. 
If $H(\mu^N | \nu_0^N) = o(N^{d-2a})$, then $H(\mu_t^N | \nu_t^N) = o(N^{d-2a})$ for all $t \in [0,T]$.
\end{thm}

While the relative entropy of probability measures on $X_N$ is usually of size $O(N^d)$, Theorem \ref{t} deals with measures constructed from fields $\bp$ which are small perturbations of $\bp_*$. To distinguish different perturbations, the relative entropy needs to be small enough. The proof of the following Corollary demonstrates that a relative entropy of size $O(N^{d-2a})$ is critical for this.

\begin{cor} \label{c:t:L1} 
Under the assumptions of Theorem \ref{t}, for any $t \in [0,T]$ and any $\bF : \T^d \to \R^{d+1}$ continuous,
\begin{equation*}
  \lim_{N \to \infty} N^{a-d} \sum_{x\in\T_N^d} \bF \Big( \frac xN \Big) \cdot [\bI(\eta_x) - \bp_*]
  = \int_{\T^d} \bF(u) \cdot \bphi(t,u) \, du
\end{equation*}
in $L^1(\mu_t^N)$.
\end{cor}

\begin{proof}
Since 
\begin{align*}
\lim_{N\to\infty}N^{-d} \sum_{x\in\T_N^d} \bF \Big(\frac xN \Big) \cdot \bphi \Big(t,\frac xN \Big) \, du = \int_{\T^d} \bF(u) \cdot \bphi(t,u) \, du,
\end{align*}
it is enough to show that
\begin{align*}
f_N(\eta) := N^{a-d} \sum_{x\in\T_N^d} \bF \Big( \frac xN \Big) \cdot \left[\bI(\eta_x) - \bp_* - \frac{1}{N^a}\bphi \Big(t,\frac xN \Big) \right]
=: 
N^{a-d} \sum_{x\in\T_N^d} \bF \Big( \frac xN \Big) \cdot \bJ_x^N(\eta_x)
\end{align*}
vanishes in $L^1(\mu_t^N)$ as $N\to\infty$.

Fix $\gamma>0$. By the entropy inequality, we get
\[
  E_{\mu_t^N} \left[ | f_N(\eta) | \right]
  \leq \frac{ H(\mu_t^N | \nu_t^N) }{\gamma N^{d-2a}} + \frac{1}{\gamma N^{d-2a}} \log E_{\nu_t^N} \left[ \exp (\gamma N^{d-2a} |f_N(\eta)|) \right].
\]
Using $e^{|x|} \le e^x+e^{-x}$, we can get rid of the absolute value in the right-hand side.
Since $\{\bF (x/N) \cdot \bJ_x^N(\eta_x) \}_x$ are independent under $\nu_t^N$, we have
\begin{multline*}
 \frac{1}{\gamma N^{d-2a}} \log E_{\nu_t^N} \left[ \exp (\gamma N^{d-2a} |f_N(\eta)|) \right] \\
 = \frac{1}{\gamma N^{d-2a}} \sum_{x\in\T_N^d} \log E_{\nu_t^N} \left[ \exp \Big(\frac{\gamma}{N^{a}} \bF \Big( \frac xN \Big) \cdot \bJ_x^N(\eta_x) \Big) \right].
\end{multline*}
Note from (recall \eqref{bP} and \eqref{nutN}) $E_{\nu_t^N}[\bI(\eta_x)] = \bp_* + N^{-a}\bphi (t, x/N )$ that $E_{\nu_t^N}[\bF (x/N) \cdot \bJ_x^N(\eta_x)] =0$ and that $\left| \bF (x/N) \cdot \bJ_x^N(\eta_x) \right| \le C$  for some constant $C>0$ which depends on $\bF$. Using the inequalities $e^x \le 1+x+ \frac12 x^2e^{|x|}$ and $ \log(1+y) \le y$, the previous expression is bounded from above by
\begin{multline*}
\dfrac{1}{\gamma N^{d-2a}} \sum_{x\in\T_N^d} E_{\nu_t^N} \left[\frac{\gamma^2}{2N^{2a}} \left( \bF \Big( \frac xN \Big) \cdot \bJ_x^N(\eta_x) \right)^2 \exp \left( \frac{\gamma}{N^a} \left| \bF \Big( \frac xN \Big) \cdot \bJ_x^N(\eta_x) \right| \right) \right] \\
 \le \frac{\gamma}{2}C^2\exp\left( \frac{\gamma}{N^a}C \right).
\end{multline*}

Collecting the estimates above, we have
\begin{align*}
  E_{\mu_t^N} \left[ | f_N(\eta) | \right]
  \leq \frac{ H(\mu_t^N | \nu_t^N) }{\gamma N^{d-2a}} + \gamma C^2\exp\left( \frac{\gamma}{N^a}C \right).
\end{align*}
Letting first $N\to\infty$ and then $\gamma\to0$ we obtain Corollary \ref{c:t:L1}.
\end{proof}

\begin{rem} \label{r:t:uniq} 
If \eqref{HDL:varphi} attains a classical solution $\bphi$, then it is the unique solution among all classical solutions with the same initial condition. Indeed, this is a standard corollary of the relative entropy method, as remarked in e.g.\ \cite[p.115]{KL}.
\end{rem}

\begin{rem} \label{r:t:genl:jumps} 
Proposition \ref{p:cV} and Theorem \ref{t} hold for more general local jumps than the nearest-neighbor jumps given by $p_N$ in \eqref{pN}. Such jumps are considered e.g.\ in \cite{AGS,Sim}. The setting for more general jumps is given by
\begin{equation*}
  (L_N^{ex} f)(\eta) 
  = \sum_{v \in \cV} \sum_{ x, y\in\T_N^d } \eta_x(v) \big( 1- \eta_y(v) \big)
p_N(y-x,v) \left\{  f\left(  \eta^{x,y,v}\right)  -f\left(\eta\right)  \right\},
\end{equation*} 
where now $p_N(z,v)$ is an irreducible probability transition function in $z$ of finite range. It decomposes as
\[
  p_N(z,v) = p^s(z) + \frac{1}{N^{1-a}} p^a(z,v),
\]
where $p^s(-z) = p^s(z)$ is symmetric and $p^a(-z,v) = -p^a(z,v)$ is anti-symmetric with
\begin{equation*}
  \sum_{z \in \Z^d} z p^a (z,v) = v.
\end{equation*}
\end{rem}

\subsection{Discussion on the admissible sets $\cV$}
\label{s:sett:disc}

In this section we review the various conditions imposed on the velocity set $\cV$ in the previously mentioned \cite{AGS,BL,EMY,QY,Sim}, and compare them to our Assumptions \ref{a:AV} and \ref{a:sg}.

In the seminal work \cite{EMY}, two key assumptions are made on $\cV$:
\begin{itemize}
\item[(a)] the dynamics is ergodic among the micro-canonical ensembles with fixed total mass and momentum;
\item[(b)] $\cV$ is invariant under reflections $R_i$ and permutations $P_{ij}$ for all $1 \le i\not=j \le d$, where for $v=(v_1,\ldots,v_d)\in\cV$,
\begin{equation*}
  R_iv:=v-2v_ie_i, \quad P_{ij}v:=v-(v_i-v_j)(e_i-e_j).
\end{equation*}
\end{itemize}
In addition, all $v\in\cV$ are required to be of equal length.
The ergodicity in (a) is necessary for the macroscopic mass-momentum vector to evolve with a closed system of equations. The coefficients in these equations can be computed explicitly due to the symmetries in $\cV$ imposed in (b).

To demonstrate the existence of non-trivial sets $\cV$ which satisfy these assumptions, \cite{EMY} provides two explicit examples:
\begin{itemize}
  \item Model I: $d\ge1$, $\cV = \{\pm e_i; i=1,\ldots,d\}$, and
  \item Model II: $d=3$, $\cV$ generated by $(1,1,\varpi)$ under all $R_i$ and $P_{ij}$, where $\varpi$ is the positive root of $\varpi^4 - 6 \varpi^2 - 1 = 0$.
\end{itemize}
In particular, the incompressible Navier--Stokes equation is recovered by Model II.

Conditions (a) and (b) have become the standard setting for hydrodynamic and incompressible limits for this model. Indeed, \cite{AGS,QY,Sim} impose the same assumptions. In \cite{BL}, condition (a) is replaced by a quantitative ergodicity, namely a spectral gap estimate.
However, it is verified only for Model I.

In the present paper, the incompressible limit is proved under Assumptions \ref{a:AV} and \ref{a:sg}.
One of our main contributions is that (b) is not required.
In addition, we extend the spectral gap estimate to a significantly wider class of $\cV$ which does not necessarily fulfill any invariance in (b) (recall Proposition \ref{p:cV}).
Hence, our incompressible limit also holds for velocity sets with strong asymmetry.
However, the technique we used to proof the spectral gap estimate relies on the following property: $\cV$ can be decomposed as a union of \textit{disjoint} pairs $(v,v')$ such that collisions can only happen between particles with velocities that belong to the same pair.
Hence, our result does not cover sets $\cV$ with more complicated collision sets, such as the one in Model II.
Nevertheless, we believe that also for such sets $\cV$ a certain polynomial spectral gap estimate holds.

\section{Proof of Theorem \ref{t}: the relative entropy method}
\label{s:ent:prod}

In this section, we prove Theorem \ref{t} by applying Yau's relative entropy method. In Section \ref{s:ent:prod:not} we introduce convenient notation. In Section \ref{s:ent:prod:rel} we perform preliminary computations. Sections \ref{s:ent:prod:comp1} and \ref{s:ent:prod:comp2} contain the main computations of the entropy production. Section \ref{s:ent:prod:BG} contains the crucial step; the application of Theorem \ref{t:BG} (a Boltzmann-Gibbs principle). We state and prove Theorem \ref{t:BG} in Section \ref{s:BG} below. 

\subsection{Notation}
\label{s:ent:prod:not} 

We always assume that $N$ is large enough with respect to the solution $\bphi$ from Theorem \ref{t} and such that $N^{2a - d} H(\mu^N | \nu_0^N) > 0$ is as small as desired.

To avoid clutter in the computations that follow, we introduce short-hand notation. First, sums over $x,y,z$ will always be understood as summing over $\T_N^d$, unless mentioned otherwise. Likewise, sums over $v,w$  always run over $\cV$ and sums over $\eta$ run over $X_N$. Second, we abbreviate several objects as in Table \ref{tab:not}. All this notation is defined from $\bphi, \bp_*, a$ (recall Theorem \ref{t} and the paragraph above it) and the functions $\bLam, \theta, \chi, \bI$ (recall respectively Proposition \ref{p:nutN}, \eqref{theta}, \eqref{HDL:rhop}, \eqref{bI}). Two useful relations between the symbols in Table \ref{tab:not} are
\begin{equation*} 
  \bp_x = \bP (\blam_x) = \sum_v \bv \theta_x,
  \qquad  \bar \bI_x = \sum_v \bv \bar \eta_x.
\end{equation*}
We stress that $\eta_x$ is short-hand notation for $\eta_x(v) \in \{0,1\}$, which differs from the meaning of $\eta_x$ in Section \ref{s:setting}.

\begin{table}[ht]
\centering
\begin{tabular}{clccccc}
\toprule
symbol & definition & \multicolumn{5}{c}{depends on} \\
 \cmidrule(lr){3-7}
 & & $t$ & $x$ & $N$ & $v$ & $\eta$ \\
\midrule
$\bphi(t, \cdot)$ & Thm.\ \ref{t} & $\bigcirc$ & $\times$ & $\times$ & $\times$ & $\times$ \\ 
 $\bphi_x$ & $\bphi(t, \tfrac xN)$ & $\bigcirc$ & $\bigcirc$ & $\bigcirc$ & $\times$ & $\times$ \\
 $\bp_x$ & $\bp_* + N^{-a} \bphi_x$ & $\bigcirc$ & $\bigcirc$ & $\bigcirc$ & $\times$ & $\times$ \\
 $\blam_x$ & $\bLam (\bp_x)$ & $\bigcirc$ & $\bigcirc$ & $\bigcirc$ & $\times$ & $\times$ \\
 $\theta_x$ & $\theta(\blam_x \cdot \bv)$ & $\bigcirc$ & $\bigcirc$ & $\bigcirc$ & $\bigcirc$ & $\times$ \\
 $\chi_x$ & $\chi(\theta_x)$ & $\bigcirc$ & $\bigcirc$ & $\bigcirc$ & $\bigcirc$ & $\times$ \\
 $\eta_x$ & $\eta_x(v)$ & ? & $\bigcirc$ & $\bigcirc$ & $\bigcirc$ & $\bigcirc$ \\
 $\bar \eta_x$ & $\eta_x - \theta_x$ & $\bigcirc$ & $\bigcirc$ & $\bigcirc$ & $\bigcirc$ & $\bigcirc$ \\
 $\bI_x$ & $\bI(\eta_x)$ & ? & $\bigcirc$ & $\bigcirc$ & $\times$ & $\bigcirc$ \\
 $\bar \bI_x$ & $\bI_x - \bp_x$ & $\bigcirc$ & $\bigcirc$ & $\bigcirc$ & $\times$ & $\bigcirc$ \\
 \bottomrule 
\end{tabular}
\caption{Short-hand notation used in Section \ref{s:ent:prod}. In chronological order the symbols are defined in terms of the objects $\bphi, \bp_*, a, \bLam, \theta, \chi, \bI$ introduced in Section \ref{s:setting}. $\eta \in X_N$ is treated as a sample; dependence on $\eta$ indicates that the symbol is a random variable. The dependence of $\eta$ on $t$ depends on how $\eta$ is distributed (e.g.\ by $\mu_\blam^N$, $\nu_t^N$ or $\mu_t^N$). \label{tab:not}}
\end{table}

\subsection{Preliminary computations}
\label{s:ent:prod:rel}

Consider the setting as in Proposition \ref{p:nutN}, where we recall 
\[
  \bLam = \bP^{-1}, \qquad 
  \bP(\blam) = \sum_v \bv \theta(\blam \cdot \bv), \qquad
  \bp_* = \bP(\bzero), \qquad
  \theta(\blam \cdot \bv) = \frac{e^{\blam \cdot \bv}}{1 + e^{\blam \cdot \bv}}.
\]

\begin{lem} \label{l:bLam:props}
The following hold:
\begin{align} \label{bLam1}
  \bLam (\bp_*) &= \bzero, \\\label{bLam3}
  \nabla \bLam (\bp_*) &= 4 \Big[\sum_v \bv \otimes \bv\Big]^{-1} = 2 A_\cV, \\\label{bLam4}
  \nabla^2 \bLam (\bp_*) &= O \in \R^{(d+1)^3}.
\end{align}
\end{lem}

\begin{proof} \eqref{bLam1} follows directly from \eqref{bpstar}. \eqref{bLam3} and \eqref{bLam4} require some computations. First, 
\begin{align} \label{vart:der}
  \theta'(\alpha) 
  &= \frac{e^{\alpha}}{(1 + e^{\alpha})^2}
  = \theta(\alpha) (1 - \theta(\alpha))
  = \chi (\theta(\alpha)), \\\notag
  \theta''(\alpha) &= (\chi \circ \theta)'(\alpha) = (1 - 2\theta(\alpha)) \theta'(\alpha).
\end{align}
In particular, we obtain
\begin{equation} \label{vart:0}
  \theta(0) = \frac12, \qquad
  \theta'(0) = \frac14, \qquad
  \theta''(0) = 0.
\end{equation}
Second, we compute
\begin{align*}
  \nabla \bP(\blam) &= \sum_v \bv \otimes \bv \frac{e^{\blam \cdot \bv}}{(1 + e^{\blam \cdot \bv})^2}, \\
  \partial_k \nabla \bP(\blam) 
  &= \sum_v \bv \otimes \bv v_k e^{\blam \cdot \bv} \frac{1 - e^{2\blam \cdot \bv}}{(1 + e^{\blam \cdot \bv})^4}.
\end{align*}
In particular, we obtain
\begin{align*}
  \bP(\bzero) &= \frac12 \sum_v \bv = \bp_*, \\
  \nabla \bP(\bzero) &= \frac14 \sum_v \bv \otimes \bv, \\
  \partial_k \nabla \bP(\bzero) &= O \in \R^{(d+1) \times (d+1)} \qquad \text{for all } k = 0,\ldots,d.
\end{align*}

Next we prove \eqref{bLam3} and \eqref{bLam4}. By using basic calculus we obtain 
\begin{equation} \label{pfzt}
  I 
  = \nabla ( \bLam \circ \bLam^{-1} )(\blam) 
  = \nabla ( \bLam \circ \bP )(\blam)
  = (\nabla \bLam (\bP (\blam)))^T \, \nabla \bP (\blam),
\end{equation} 
where $I$ is the identity matrix and $|\blam|$ is small enough such that $\bP(\blam)$ is in the domain of $\bLam$. Substituting $\blam = 
\bzero$, we obtain \eqref{bLam3} from
\[
  \nabla \bLam (\bp_*) 
  = \nabla \bLam (\bP(\bzero))
  = [\nabla \bP (\bzero)]^{-T}.
\]
To prove \eqref{bLam4}, take any $k$. Taking $\partial_k$ in \eqref{pfzt} yields
\[
  O 
  = \sum_{l=0}^d \partial_k P_l (\blam) \big( (\partial_l \nabla \bLam) (\bP (\blam)) \big)^T \nabla \bP (\blam)
    + (\nabla \bLam (\bP (\blam)))^T \, \partial_k \nabla \bP (\blam)
\]
for $|\blam|$ small enough.
Substituting $\blam = \bzero$ and using $\partial_k \nabla \bP (\bzero) = O$  and that $\nabla \bP (\bzero)$ is an invertible matrix, we obtain
\begin{align} \label{pfyg}
  O 
  = \sum_{l=0}^d \partial_k P_l (\bzero) (\partial_l \nabla \bLam) (\bp_*).
\end{align}
Since $\nabla \bP (\bzero)$ is invertible, its column vectors $u_k := \partial_k \bP (\bzero)$ form a basis of $\R^{d+1}$. Then, interpreting $\nabla^2 \bLam (\bp_*)$ as a linear map from $\R^{d+1}$ to $\R^{(d+1) \times (d+1)}$, it follows from \eqref{pfyg} that the kernel of $\nabla^2 \bLam (\bp_*)$ is the full space $\R^{d+1}$. \eqref{bLam4} follows.
\end{proof}

Next, we prepare several expansions. Let $t,x,v$ be given. Here and henceforth, the $O(N^{\alpha})$-terms will always be uniform in $t,x$.
First, we use Lemma \ref{l:bLam:props} to expand \comm{[uniformity Checked on 5Nov]}
\begin{align} \label{pfyf}
  \blam_x
  &= \bLam(\bp_* + N^{-a} \bphi_x)
  = 2 N^{-a} A_\cV \bphi_x + O(N^{-3a}), \\ \label{pfzb}
  \nabla \bLam(\bp_x)
  &= \nabla \bLam(\bp_* + N^{-a} \bphi_x)
  = 2 A_\cV + O(N^{-2a}).
\end{align}
Second, we write $\theta_y = \theta (\bLam (\bp_y) \cdot \bv) =: g(\bp_y)$ and expand $g$ around $\bp_x$. Noting that $g(\bp_x) = \theta_x$ and 
\[
  \nabla g(\bp_x) 
  = \theta'(\blam_x \cdot \bv) (\nabla \bLam(\bp_x))^T \bv
  = \chi_x (\nabla \bLam(\bp_x))^T \bv,
\]  
we obtain (using \eqref{pfzb} and $a < \frac12$)
\begin{align} \notag
  \theta_y - \theta_x
  &= \chi_x \bv \cdot \nabla \bLam(\bp_x) (\bp_y - \bp_x) + O(|\bp_y - \bp_x|^2) \\ \label{pfzg}
  &= N^{-a}\chi_x \bv \cdot \nabla \bLam(\bp_x) (\bphi_y - \bphi_x) + O(N^{-2a} |\bphi_y - \bphi_x|^2) \\ \label{pfyk}
  &= 2 N^{-1-a}\chi_x \bv \cdot A_\cV \nabla \bphi_x (y-x) + O(N^{-1-3a}) \\ \label{pfyy}
  &= O(N^{-1-a}).
\end{align}
Third, recalling \eqref{vart:0} and using \eqref{pfyf} we obtain
\begin{align*}
  \theta_x 
  &= \theta (\blam_x \cdot \bv)
  = \theta(0) + \theta'(0) \blam_x \cdot \bv + \frac12 \theta''(0) (\blam_x \cdot \bv)^2 + O(N^{-3a}) \\
  &= \frac12 + \frac{N^{-a}}2 \bv \cdot A_\cV \bphi_x + O(N^{-3a}),
\end{align*}
and thus
\begin{align} \label{pfzf}
  1 - 2\theta_x
  &= - {N^{-a}} \bv \cdot A_\cV \bphi_x + O(N^{-3a}), 
  \\ \label{pfze}
  \chi_x 
  &= \theta_x (1-\theta_x) 
  = \frac14 + O(N^{-2a}). 
\end{align}

\subsection{Entropy production}
\label{s:ent:prod:comp1}

To estimate the entropy production, we apply Yau's entropy inequality: 
\begin{equation} \label{pfzu} 
  \frac d{dt}  H(\mu_t^N | \nu_t^N)
  \leq \int L_N f_t^N \, d\nu_t^N - \int \partial_t \log \nu_t^N \, d\mu_t^N,
  \qquad f_t^N := \frac{d \mu_t^N}{d \nu_t^N}.
\end{equation}
This inequality is classical; see e.g.\ \cite[p.\ 120 and p.\ 342]{KL}. 
We compute and estimate both terms in the right-hand side of \eqref{pfzu} separately. The goal is to obtain 
\begin{multline} \label{pfzq}
  \int_0^T \frac d{dt}  H(\mu_t^N | \nu_t^N) \, dt
  \leq
  N^{-a} \int_0^T  \sum_x  E_{\mu_t^N} \left[\bar \bI_x\right] \cdot A_\cV \cK [\bphi] \Big(t,\frac xN \Big) \, dt  \\
   + C \int_0^T  H(\mu_t^N | \nu_t^N) \, dt + O(N^{d-2a-\e_0}),
\end{multline}
where $\e_0 > 0$ is a constant and $\cK [\bphi](t, \frac xN) = \bzero \in \R^{d+1}$ is short-hand notation for the incompressible limit \eqref{HDL:varphi}. Theorem \ref{t} follows from \eqref{pfzq} by a standard application of Gronwall's Lemma. 

We start with the $\partial_t$-term in \eqref{pfzu}. Since $\nu_t^N = \mu_{\blam(t, \cdot)}^N$ is the Bernoulli product measure on $\{0,1\}^{\T_N ^d \times \cV}$ with varying parameter $\theta_x$, a simple computation (see e.g.\ \cite[Lemma 4.6]{Fun}) shows that
\begin{align*}  
  \partial_t \log \nu_t^N (\eta)
  &= \sum_{xv} \partial_t \theta_x \frac{\bar \eta_x}{\chi_x}.
\end{align*}
Computing the time derivative, we get (recall \eqref{vart:der}) 
\[
  \partial_t \theta_x 
  = \theta'(\blam_x \cdot \bv) \bv \cdot \partial_t \blam_x
  = N^{-a }\chi_x \bv \cdot \nabla \bLam (\bp_x) \partial_t \bphi_x.
\]
Hence, using \eqref{pfzb} 
\begin{align*} 
  \int \partial_t \log \nu_t^N \, d \mu_t^N
  &= N^{-a} \sum_x E_{\mu_t^N} \bigg[ \sum_{v} \bar \eta_x \bv \bigg] \cdot  \nabla \bLam (\bp_x) \partial_t \bphi_x
\\
  &= 2 N^{-a} \sum_{x}  E_{\mu_t^N} \left[ \bar \bI_x \right] \cdot A_\cV \partial_t \bphi_x + O(N^{d-3a}),  
\end{align*}
which, apart from the error term, is part of the integrand of the first term on the right-hand side of \eqref{pfzq}. 
 
Next we treat the $L_N$ part in \eqref{pfzu}. Similar to \eqref{mulam:bI} it follows that $\nu_t^N$ is a function of $(\bI(\eta_x))_x$. Then, it follows from Lemma \ref{l:LNc:0} that $L_N^c \nu_t^N = 0$. Hence,
\begin{equation} \label{pfza}
  \int L_N f_t^N \, d\nu_t^N
  = N^2 \sum_{\eta xyv} \eta_x (1 - \eta_y) p_N(y-x) \big[ f_t^N(\eta^{x,y,v}) - f_t^N(\eta) \big] \nu_t^N (\eta).
\end{equation}
Unlike the $\partial_t$-term, we cannot rewrite this directly into the desired form \eqref{pfzq}. Instead, our strategy is first to write it in the form
\begin{equation} \label{pfyz}
  \sum_{ v} E_{\mu_t^N} \left[ \sum_{ x} \eta_x a_{t,x,v}^N \right] + \sum_{v, |z|=1} E_{\mu_t^N} \left[\sum_{ x} \eta_x \eta_{x+z} a_{t,x,z,v}^N \right]
\end{equation}
for deterministic coefficients $a_{t,x,v}^N, a_{t,x,z,v}^N \in \R$ which are uniformly bounded, and then apply the Boltzmann--Gibbs principle in Theorem \ref{t:BG} to replace $\eta_x$ and $\eta_x \eta_{x+z}$ by terms which only depend on $\eta$ through $\bar \bI_x$. This replacement will produce the error term given by the latter two terms in the desired estimate \eqref{pfzq}. 

With this plan in mind, we first rewrite \eqref{pfza} such that all appearances of $f_t^N$ and $\nu_t^N$ can be replaced by $\mu_t^N = f_t^N \nu_t^N$.
Regarding the part of the summand involving $f_t^N(\eta^{x,y,v})$, we change variables $\zeta = \eta^{x,y,v}$ and swap $x$ and $y$. This yields
\begin{equation*} 
  \int L_N f_t^N \, d\nu_t^N
  = N^2 \sum_{\eta xyv} \eta_x (1 - \eta_y)  \big[ p_N(x-y) \nu_t^N(\eta^{x,y,v}) - p_N(y-x) \nu_t^N(\eta) \big] f_t^N (\eta).
\end{equation*}
Using that $\eta_x (1 - \eta_y)$ vanishes unless $\eta_x = 1$ and $\eta_y = 0$, and that $\nu_t^N$ is the Bernoulli product measure on $\{0,1\}^{\T_n^d \times \cV}$ with varying parameter $\theta_x$, we rewrite
\begin{equation*} 
  \eta_x (1 - \eta_y) \nu_t^N(\eta^{x,y,v})
  = \eta_x (1 - \eta_y) \frac{\theta_y (1 - \theta_x)}{\theta_x (1 - \theta_y)} \nu_t^N(\eta).
\end{equation*}
Then, substituting $p_N(x-y) = 1 + N^{a-1} v \cdot (x-y)$ we get
\begin{align} \notag 
  &\int L_N f_t^N \, d\nu_t^N \\\notag
  &= N^2 \sum_{\eta xyv} \frac{ \eta_x (1 - \eta_y) }{\theta_x (1 - \theta_y)} \underbrace{ \Big( {\theta_y - \theta_x} + N^{a-1} (v \cdot  (x-y))(\theta_x + \theta_y - 2 \theta_x \theta_y) \Big) }_{ =: b_{x,y}} \underbrace{ \nu_t^N(\eta) f_t^N (\eta) }_{\mu_t^N(\eta)} \\\label{pfzl}
  &= \sum_{v} E_{\mu_t^N} \left[ \sum_{ x} \eta_x \Big( N^2 \sum_y \frac{b_{x,y}}{\theta_x (1 - \theta_y)} \Big) \right] - \sum_{v} E_{\mu_t^N} \left[ \sum_{ xy} N^2\frac{\eta_x \eta_y}{\theta_x (1 - \theta_y)} b_{x,y} \right].
\end{align}
While both terms are already in the desired form \eqref{pfyz}, it turns out convenient to rewrite the latter. We exploit the symmetry in $\sum_{xy}$ to remove summands that are anti-symmetric in $x,y$. The product $\eta_x \eta_y$ is symmetric and $b_{x,y}$ is anti-symmetric. We split the remaining part into its symmetric and anti-symmetric parts as 
\[
  \frac1{\theta_x (1 - \theta_y)} 
  = \frac{\theta_y (1 - \theta_x)}{\chi_x \chi_y} 
  = \frac{\theta_y + \theta_x - 2 \theta_y \theta_x}{2 \chi_x \chi_y} + \frac{\theta_y - \theta_x}{2 \chi_x \chi_y}.
\]
Hence,
\begin{equation*}
  \sum_{ xy } \frac{\eta_x \eta_y}{\theta_x (1 - \theta_y)} b_{x,y}
  = \sum_{ xy } \eta_x \eta_y \frac{ (\theta_y - \theta_x) b_{x,y} }{2 \chi_x \chi_y}
\end{equation*}
and thus
\begin{equation} \label{pfyw}
  \int L_N f_t^N \, d\nu_t^N
  = \sum_v E_{\mu_t^N} \left[ \sum_{ x} \eta_x a_x \right] + \sum_{v, |z|=1} E_{\mu_t^N} \left[ \sum_{ x} \eta_x \eta_{x+z} a_{x,z} \right], 
\end{equation}
where
\begin{equation} \label{pfxm}
  a_x := N^2 \sum_y \frac{b_{x,y}}{\theta_x (1 - \theta_y)},
  \qquad a_{x,z} := \frac{ N^2 (\theta_x - \theta_{x+z}) b_{x,x+z} }{2 \chi_x \chi_{x+z}}.
\end{equation}

\subsection{Application of the Boltzmann-Gibbs principle} 
\label{s:ent:prod:BG}

Integrating \eqref{pfyw} over $t$ yields
\begin{equation} \label{pfyb}
  \int_0^T \int L_N f_t^N \, d\nu_t^N \, dt
  = \sum_v \E_{\mu^N} \left[  \int_0^T \sum_{ x} a_x \eta_x  \, dt \right] + \sum_{v, |z|=1} \E_{\mu^N} \left[ \int_0^T \sum_{ x} a_{x,z} \eta_x \eta_{x+z} \, dt \right].
\end{equation}

The Boltzmann--Gibbs principle in Theorem \ref{t:BG} states that for any local function $h:(\{0,1\}^{\cV})^{\Z^d}\to\R$ and for any uniformly bounded deterministic coefficients $a_x' = a_x'(t,N)$ we have 
\begin{align} \label{BG:summary}
  \E_{\mu^N} \left[ \int_0^T \sum_{ x} a_x' \tau_x h(\eta) \, dt \right]
  = \int_0^T  \sum_{ x} a_x' \tilde h(\bp_x) \,dt
   + \int_0^T  \sum_{ x} a_x' \nabla \tilde h(\bp_{t,x}) \cdot E_{\mu_t^N} \left[  \bar \bI_x \right] \,dt + \e_{BG},
\end{align} 
where $\tilde h(\bp) := E_{\nu_{\bp}^N}[h]$, $\tau_x$ is the usual shift operator with $\tau_x h (\eta) := h (\tau_x \eta)$ and $(\tau_x \eta)_y := \eta_{x+y}$, and 
\[
  \e_{BG} := C_0 \int_0^T H(\mu_t^N|\nu_t^N) \,dt + O (N^{d-2a-\e_0})
\]
is the error term, where $\e_0 \in (0,a]$ and $C_0 > 0$ are constant; $\e_0$ only depends on $d,a,\kappa$ and $C_0$ is independent of $N$. To apply \eqref{BG:summary} to \eqref{pfyb} we take $a_x$ and $a_{x,z}$ as the coefficients $a_x'$ and
\begin{align*}
  h(\eta) := \eta_0(v) \quad \text{and} \quad h_z(\eta) := \eta_0(v) \eta_z(v)
 \end{align*}  
 as the corresponding local functions. This requires that there exists a constant $C > 0$ such that
\begin{equation} \label{pfyx}
  |a_x| + |a_{x,z}| \leq C \qquad \text{for all } x,z,v,t,N,
\end{equation}
which we prove at the end of this subsection. 

As preparation we compute for a general variable $\bp$ close enough to $\bp_*$ (recall the product structure of $\mu_\blam^N$ in \eqref{mulam:prod})
\begin{align*}
   \tilde h(\bp)
   &= E_{\mu_{\bLam(\bp)}^N} [\eta_0(v)]
   = \frac{ e^{\bLam(\bp) \cdot \bv} }{1 + e^{\bLam(\bp) \cdot \bv}} 
   = \theta (\bLam(\bp) \cdot \bv ), \\
   \nabla \tilde h(\bp)
   &= \theta' (\bLam(\bp) \cdot \bv ) (\nabla \bLam(\bp))^T \bv. 
\end{align*} 
Substituting $\bp = \bp_x$ we get (recall \eqref{vart:der})
\begin{align*}
   \tilde h(\bp_x)
   = \theta_x, \qquad
   \nabla \tilde h(\bp_x) 
   =  \chi_x (\nabla \bLam(\bp_x))^T \bv. 
 \end{align*} 
Similarly, 
\begin{align*}
   \tilde h_z(\bp)
   &= E_{\mu_{\bLam(\bp)}^N} [\eta_0(v) \eta_z(v)]
   = \tilde h(\bp)^2
   = \theta (\bLam(\bp) \cdot \bv )^2, \\
   \nabla \tilde h_z(\bp)
   &= 2 \theta (\bLam(\bp) \cdot \bv ) \theta' (\bLam(\bp) \cdot \bv ) (\nabla \bLam(\bp))^T \bv. 
\end{align*} 
Substituting $\bp = \bp_x$ we get
\begin{align*}
   \tilde h_z(\bp_x)
   = \theta_x^2, \qquad
   \nabla \tilde h_z(\bp_x) 
   = 2 \theta_x \chi_x (\nabla \bLam(\bp_x))^T \bv. 
 \end{align*} 
Note that both expressions are independent of $z$. 
 
With these preparations we apply \eqref{BG:summary} to \eqref{pfyb}. This yields
\begin{multline} \label{pfyv}
  \int_0^T \int L_N f_t^N \, d\nu_t^N \, dt
  = \int_0^T  \sum_{xv}  \Big( \theta_x a_x  + \theta_x^2 \sum_{|z|=1} a_{x,z} \Big) dt \\
    + \int_0^T  \sum_{xv} \bv \cdot \nabla \bLam(\bp_x) E_{\mu_t^N} \left[ \bar \bI_x  \right] \chi_x \Big( a_x + 2 \theta_x \sum_{|z|=1} a_{x,z} \Big) \,dt + \e_{BG}
\end{multline}
for a different, larger constant $C_0$ in the expression of $\e_{BG}$.

It is left to prove the a priori bound \eqref{pfyx} on the coefficients $a_x$, $a_{x,z}$ (recall \eqref{pfxm}). From \eqref{pfyy} we obtain $a_{x,z} = O(1)$. To see that $a_x = O(1)$, we use the identity
\[  
  \frac1{1-\theta_y} - \frac1{1-\theta_x}
  = \frac{\theta_y - \theta_x}{(1-\theta_y)(1-\theta_x)}.
\]
to replace $1-\theta_y$ in the denominator in \eqref{pfxm}. This yields
\begin{equation*}
  a_x
  = N^2 \sum_y \frac{b_{x,y}}{\chi_x} \Big( 1 + \frac{\theta_y - \theta_x}{1-\theta_y} \Big).
\end{equation*}
As we saw for $a_{x,z}$, the second part in the parentheses yields an $O(1)$ contribution to $a_x$. Writing out $b_{x,y}$, we get
\begin{equation*}
  a_x
  = \frac{N^2}{\chi_x} \sum_y (\theta_y - \theta_x) + \frac{N^{1+a}}{\chi_x} \sum_y (v \cdot  (x-y))(\theta_x + \theta_y - 2 \theta_x \theta_y) + O(1).
\end{equation*}
For the first part, we obtain from \eqref{pfzg} that
\begin{equation*}
  \sum_y (\theta_y - \theta_x)
  = N^{-a}\chi_x \bv \cdot \nabla \bLam(\bp_x) \sum_y (\bphi_y - \bphi_x) + O(N^{-2-2a})
  = O(N^{-2-a}).
\end{equation*}
Then, changing variables to $z = x-y$ and using that $v \cdot z$ is odd in $z$, we obtain
\begin{align*}
  a_x
  &= \frac{N^{1+a}}{\chi_x} \sum_z (v \cdot z)(\theta_x + \theta_{x-z}(1 - 2 \theta_x )) + O(1) \\
  &= \frac{N^{1+a}}{2\chi_x} \sum_z (v \cdot z)(\theta_{x-z} - \theta_{x+z})(1 - 2 \theta_x) + O(1),
\end{align*}
which by \eqref{pfzg} is $O(1)$. 
 
\subsection{Continuation of the entropy production computation} 
\label{s:ent:prod:comp2}
 
We continue the computation from \eqref{pfyv}. The integrand of the first integral in the right-hand side equals (using again the symmetry in the sum $\sum_{xy}$ to cancel out the anti-symmetric part of the summand)
\begin{align*}
  &N^2 \sum_{ xyv} \frac{b_{x,y}}{2 \chi_x \chi_y} \big( 2 \theta_x \theta_y (1-\theta_x) + \theta_x^2 (\theta_x - \theta_y) \big) \\
  &= N^2 \sum_{ xyv} \frac{b_{x,y}}{2 \chi_x \chi_y} \Big( \theta_x \theta_y (\theta_y-\theta_x) + \frac{\theta_x^2 + \theta_y^2 }2 (\theta_x - \theta_y) \Big) \\
  &= N^2 \sum_{ xyv} \frac{b_{x,y}}{4 \chi_x \chi_y} (\theta_x - \theta_y)^3, 
\end{align*}
which by \eqref{pfyy} is $O(N^{d-1-3a})$. Hence,
\begin{multline} \label{pfxn}
  \int_0^T \int L_N f_t^N \, d\nu_t^N \,dt \\
  = \int_0^T  N^2  \sum_{ xv} \bv \cdot \nabla \bLam(\bp_x) E_{\mu_t^N} \left[ \bar \bI_x \right] \sum_y b_{x,y} \Big( \frac{1 - \theta_x}{1 - \theta_y} + \theta_x \frac{ \theta_x - \theta_y }{\chi_y} \Big)  \,dt + O(N^{d-3a}) + \e_{BG}.
\end{multline}
Note for the summand over $y$ that
\begin{align} \notag
  \frac{1 - \theta_x}{1 - \theta_y} + \theta_x \frac{ \theta_x - \theta_y }{ \chi_y} =  1 + \frac{(\theta_x - \theta_y)^2 }{ \chi_y}.
\end{align}
Since by \eqref{pfyy} we have $
  b_{x,y} = O(N^{a-1})
$,
it follows from $a \leq \frac 12$ that the integrand in the right-hand side of \eqref{pfxn} equals
\begin{equation} \label{pfzk}
  N^2 \sum_{ xv} \bv \cdot \nabla \bLam(\bp_x) E_{\mu_t^N} \left[ \bar \bI_x \right] \sum_y b_{x,y} + O(N^{d-3a}).
\end{equation}

Next, we substitute the expression for $b_{x,y}$ (recall \eqref{pfzl}), and split the sum into the two parts corresponding to the two parts of $b_{x,y}$. For the part corresponding to $(\theta_y - \theta_x)$, we obtain from $\sum_v \bv \theta_x = \bLam^{-1}(\blam_x) = \bp_x$ and the observation that $\bar \bI_x$ and $\bp_x$ do not depend on $v$ that
\begin{align*}
  N^2 \sum_{xyv} 
  (\theta_y - \theta_x) \bv \cdot \nabla \bLam(\bp_x) \left[ E_{\mu_t^N} \bar \bI_x \right]
  &= N^2 \sum_{xy} 
  (\bp_y - \bp_x) \cdot \nabla \bLam(\bp_x) \left[ E_{\mu_t^N} \bar \bI_x \right] \\
  &= N^{-a} \sum_{x} 
  \Delta_N \bphi_x \cdot \nabla \bLam(\bp_x) \left[ E_{\mu_t^N} \bar \bI_x \right].
\end{align*}
Since $\bphi$ is $3$ times continuously differentiable in space, we have $\Delta_N \bphi_x = \Delta \bphi_x + O(N^{-1})$. Together with $a \leq \frac12$ and the expansion \eqref{pfzb}, we obtain
\begin{equation*}
  N^{-a} \sum_{x} 
  \Delta_N \bphi_x \cdot \nabla \bLam(\bp_x) \left[ E_{\mu_t^N} \bar \bI_x \right]
  = 2 N^{-a} \sum_{x} 
  \Delta \bphi_x \cdot A_\cV \left[ E_{\mu_t^N} \bar \bI_x \right] + O(N^{d- 3a}).
\end{equation*}

For the remaining part of \eqref{pfzk} we use $\sum_y (x-y) = 0$ to obtain
\begin{multline} \label{pfzc}
  N^{a+1} \sum_{xv} \big( \bv \cdot \nabla \bLam(\bp_x) \left[ E_{\mu_t^N} \bar \bI_x \right] \big) \sum_y (v \cdot (x-y)) (\theta_x + \theta_y - 2 \theta_x \theta_y)  \\
  = N^{a+1} \sum_{xv} \big( \bv \cdot \nabla \bLam(\bp_x) \left[ E_{\mu_t^N} \bar \bI_x \right] \big) (1 - 2 \theta_x) \sum_y (v \cdot (x-y))  (\theta_y - \theta_x).
\end{multline}
Using \eqref{pfyk} we obtain that the sum over $y$ reads as
\begin{align*}  
  \sum_y (v \cdot (x-y))  (\theta_y - \theta_x)
  &= 2 N^{-1-a}\chi_x \sum_y (v \cdot (x-y))  \bv \cdot A_\cV \nabla \bphi_x (y-x) + O(N^{-1-3a}) \\
  &= - 4 N^{-1-a} \chi_x \bv \cdot A_\cV \nabla \bphi_x v + O(N^{-1-3a}).
\end{align*}
Substituting this and the expansions \eqref{pfzb}, \eqref{pfzf} and \eqref{pfze}, Equation \eqref{pfzc} reads as
\begin{align*}
  2 N^{-a} \sum_{xv}  (\bv \cdot A_\cV \bphi_x) (\bv \cdot A_\cV \nabla \bphi_x v) (\bv \cdot A_\cV \left[ E_{\mu_t^N} \bar \bI_x \right]) 
    + O(N^{d-3a}).  
\end{align*}

In conclusion, substituting all obtained estimates and expansions in \eqref{pfzu} we obtain
\begin{align*} 
  &\int_0^T \frac d{dt}  H(\mu_t^N | \nu_t^N) \,dt \\
  &\leq 2 N^{-a} \int_0^T \sum_x E_{\mu_t^N} \left[ \bar \bI_x \right] \cdot A_\cV \bigg[ -\partial_t \bphi_x + \Delta \bphi_x 
  + \sum_v (\bv \cdot A_\cV \bphi_x) \big( \bv \cdot A_\cV \nabla \bphi_x v \big) \bv \bigg] \,dt \\ 
  &\quad + C_0 \int_0^T H(\mu_t^N | \nu_t^N) \,dt + O(N^{d-2a-\e_0}).
\end{align*}
This is precisely the desired estimate \eqref{pfzq}. This completes the proof of Theorem \ref{t}.

\section{Boltzmann-Gibbs principle} 
\label{s:BG}

Consider the setting from Theorem \ref{t}. We adopt the notation $\bp_x$, $\eta_x$, $\bI_x$ and $\bar \bI_x$ from Table \ref{tab:not}. For a local function $h:(\{0,1\}^{\cV})^{\Z^d}\to\R$, let
\begin{align} \label{fx}
f_{t,x}(\eta) := \tau_xh(\eta) - \tilde h(\bp_x) - \nabla \tilde h(\bp_x) \cdot \bar \bI_x,
\end{align}
where we recall that $\bp_x$ and $\bar \bI_x$ depend on $t$ and that $\tau_x$ is the shift operator with $\tau_x h (\eta) = h (\tau_x \eta)$ and $(\tau_x \eta)_y = \eta_{x+y}$. The function $\tilde h$ is defined in a neighborhood of $\bp_* \in \R^{d+1}$ by 
\begin{align*} 
\tilde h(\bp) = E_{\nu_{\bp}^N}[h].
\end{align*}

\begin{thm} \label{t:BG}
Let the setting and assumptions be as in Theorem \ref{t}. Let $h$ and $f_{t,x}$ be as above. Let $a_{t,x}^N  \in\R$ for $t\in[0,T]$, $x\in\T_N^d$ and $N \geq 1$ be deterministic coefficients which satisfy
\begin{align*} 
\sup_{N \geq 1} \sup_{t\in[0,T]} \max_{x\in\T_N^d}|a_{t,x}^N| < \infty.
\end{align*}
Then, there exist constants $C_0, \e_0 > 0$ such that
for all $N$ large enough
\begin{align*}
\E_{\mu^N}\left[ \left| \int_0^T \sum_{x\in\T_N^d} a_{t,x}^N f_{t,x} dt \right| \right] \le 
C_0 \left( \int_0^T H(\mu_t^N|\nu_t^N) dt + N^{d-2a-\e_0}\right).
\end{align*}
\end{thm}

For an explicit expression of $\e_0$ see \eqref{pfyh} below.

The proof of Theorem \ref{t:BG} is given at the end of this section. It is based on the two key Lemmas \ref{l:BG:mx} and \ref{l:BG:box}. In preparation for stating these lemmas, we simply write $f_x = f_{t,x}$ and $a_x = a_{t,x}^N$ because the dependence of these objects on $t,N$ is of little importance in the computations that follow. We denote by $C > 0$ a generic constant which is independent of $N$. Its value may change from line to line, but within a display it remains the same. To separate several generic constants in the same computation we use $C', C''$ etc.

We split 
\begin{multline} \label{pfyt}
  \E_{\mu^N} \left[ \bigg |\int_0^T \sum_{x \in \T_N^d} a_x f_x dt \bigg | \right] \\
  \leq \E_{\mu^N} \left[ \bigg |\int_0^T \sum_{x \in \T_N^d} a_x m_x dt \bigg | \right]
  + \E_{\mu^N} \left[ \bigg |\int_0^T \sum_{x \in \T_N^d} a_x E_{\nu_{\bp_*}^N}[f_x \mid \bI_x^M] dt \bigg | \right],  
\end{multline}  
where $M = M(N)$ is an integer,
\begin{align*}
  &m_x := f_x - E_{\nu_{\bp_*}^N}[f_x \mid \bI_x^M], \quad
  \bI_x^M := \frac1{|\Lambda_M|} \sum_{y \in {\Lambda_{x,M}}} \bI_y,\\
  &\Lambda_M := \{ y \in \T_N^d : |y| \leq M \}, \quad
  {\Lambda_{x,M}} := \{ y \in \T_N^d : y - x \in \Lambda_M \}.
\end{align*}
We choose $M$ at the end of the proof; during the proof we use
\comm{[Needed for Rayleigh, and for $\gamma \gg 1$ in Lemma \ref{l:BG:mx}]}
\begin{equation} \label{l:L10:6:assn}
   1 \ll M^{\kappa + d} \ll \min \big\{ N^{2 - 2a}, N^{\kappa + d} \big\}
  \quad \text{as } N \to \infty.
\end{equation}
Note that $\bI_x^M$ is a local average of $\bI_x$. We estimate both terms in \eqref{pfyt} separately in Lemmas \ref{l:BG:mx} and \ref{l:BG:box}.

\begin{lem} \label{l:BG:mx}
For all $N$ large enough
  \[
      \E_{\mu^N} \left[ \bigg |\int_0^T \sum_{x \in \T_N^d} a_x m_x dt \bigg | \right]
      \leq C \frac{ M^{\tfrac{\kappa + d}2} }{N^{1-a}} N^{d-2a},
  \]
where the constant $C > 0$ is independent of $N$.
\end{lem} 

\begin{proof} 
We apply the entropy inequality with respect to $\nu_{\bp_*}^N$ and with parameter 
\begin{equation} \label{pfyj}
   1 \ll \gamma \ll \frac{N^2}{M^{\kappa + d}}.
\end{equation} 
This yields
\begin{equation} \label{pfyl}
  \E_{\mu^N} \left[ \bigg |\int_0^T \sum_{x \in \T_N^d} a_x m_x dt \bigg | \right]
  \leq \frac1\gamma H(\P_{\mu^N} | \P_{\nu_{\bp_*}^N}) + \frac1\gamma \log \E_{\nu_{\bp_*}^N} \bigg[ \exp \bigg| \gamma \int_0^T \sum_{x \in \T_N^d} a_x m_x  dt \bigg| \bigg].
\end{equation} 
For the first term, we write \comm{[PvM: I use here pdf QandA with KT p.9-10]}
\begin{equation} \label{pfyn}  
  H(\P_{\mu^N} | \P_{\nu_{\bp_*}^N})
  = H(\mu^N | \nu_{\bp_*}^N)
  = H(\mu^N | \nu_0^N) + \int \log \frac{d \nu_0^N}{d \nu_{\bp_*}^N} \, d\mu^N.
\end{equation}
To the second term in \eqref{pfyn} we apply the entropy inequality. This yields 
\[
  \int \log \frac{d \nu_0^N}{d \nu_{\bp_*}^N} \, d\mu^N
  \leq \frac12 H(\mu^N | \nu_{\bp_*}^N) + \frac12 \log \int \exp ( 2 \log \frac{d \nu_0^N}{d \nu_{\bp_*}^N} ) d \nu_{\bp_*}^N.
\]
Substituting this in \eqref{pfyn} and rearranging terms, we obtain
\begin{equation} \label{pfym}  
  H(\mu^N | \nu_{\bp_*}^N)
  \leq 2 H(\mu^N | \nu_0^N) + \log \int \Big( \frac{d \nu_0^N}{d \nu_{\bp_*}^N} \Big)^2 d \nu_{\bp_*}^N.
\end{equation}
We claim that the right-hand side is $O(N^{d-2a})$. Using this claim we obtain
\[
  \frac1\gamma H(\P_{\mu^N} | \P_{\nu_{\bp_*}^N}) \leq \frac C\gamma N^{d-2a}.
\]
To prove this claim, we treat both terms on the right-hand side of \eqref{pfym} separately. The first term is given to be $O(N^{d-2a})$, and for the second term this follows from the following computation. Note that for all $\eta \in X_N$ (recall \eqref{mulam:prod}) 
\[
  \frac{\nu_0^N(\eta)^2}{\nu_{\bp_*}^N(\eta)}
  = \frac{\mu_{\blam(0,x)}^N(\eta)^2}{\mu_\bzero^N(\eta)}
  = \prod_{x \in \T_N^d} \prod_{v \in \cV} 2 \frac{ e^{2 \blam(0,x) \cdot \bv \eta_x } }{ (e^{\blam(0,x) \cdot \bv} + 1)^2 }.
\]
Hence,
\begin{align*}
  \int \Big( \frac{d \nu_0^N}{d \nu_{\bp_*}^N} \Big)^2 d \nu_{\bp_*}^N
  = \sum_{\eta \in X_N} \prod_{x \in \T_N^d} \prod_{v \in \cV} 2 \frac{ e^{2 \blam(0,x) \cdot \bv \eta_x } }{ (e^{\blam(0,x) \cdot \bv} + 1)^2 }
  = \prod_{x \in \T_N^d} \prod_{v \in \cV} 2 \frac{ e^{2 \blam(0,x) \cdot \bv } + 1 }{ (e^{\blam(0,x) \cdot \bv} + 1)^2 }.
\end{align*}
By \eqref{pfyt} we have $|\blam(0,x) \cdot \bv| \leq C N^{-a}$. Then, using the expansion 
\[
  2\frac{x^2 + 1}{(x+1)^2} = 1 + O (x-1)^2,
\]
we obtain
\begin{align*}
  \log \int \Big( \frac{d \nu_0^N}{d \nu_{\bp_*}^N} \Big)^2 d \nu_{\bp_*}^N
  \leq \log \prod_{x \in \T_N^d} \prod_{v \in \cV} (1 + CN^{-2a})
  = |\cV| N^d \log(1 + CN^{-2a}) \leq C' N^{d-2a}.
\end{align*}
This proves the claim that the right-hand side in \eqref{pfym} is $O(N^{d-2a})$.

For the second term in \eqref{pfyl} we use $e^{|x|} \leq e^x + e^{-x}$, and then apply the Feynman--Kac formula (see \cite[Appendix 1, Lemma 7.2]{KL}). This yields   \comm{$\max_g$ would also be OK, but no one does this in the literature. Indeed, $g = (g(\eta))_\eta$ is a vector in $\R^{|X_N|}$ and $0 \leq g(\eta) \leq 1 / \min_\zeta \nu_{\bp_*}^N(\zeta) < \infty$ for all $\eta$} 
\begin{multline*}
  \frac1\gamma \log \E_{\nu_{\bp_*}^N} \bigg[ \exp \bigg| \gamma \int_0^T \sum_{x \in \T_N^d} a_x m_x  dt \bigg| \bigg] \\
  \leq \max_\pm \frac1\gamma \int_0^T \sup_g \Big( \gamma \sum_{x \in \T_N^d} \langle \pm a_x m_x, g \rangle_{\nu_{\bp_*}^N} - \langle -L_N \sqrt g, \sqrt g \rangle_{\nu_{\bp_*}^N} \Big) dt + \frac{\log 2}\gamma,
\end{multline*}
where the supremum is carried over all functions $g:X_N\to[0,\infty)$ satisfying $\int_{X_N} gd\nu_{\bp_*}^N=1$.  
Since $\gamma \gg 1$ we may absorb $\gamma^{-1} \log 2$ in the constant $C$ in Lemma \ref{l:BG:mx}; we neglect it in the remainder. We also neglect $\max_\pm$, because the proof below works verbatim when $m_x$ is replaced by $-m_x$. Hence, it is left to bound
\begin{equation} \label{pfyq} 
  \frac1\gamma \int_0^T \sup_g \Big( \gamma \sum_{x \in \T_N^d} \langle a_x m_x, g \rangle_{\nu_{\bp_*}^N} - \langle -L_N \sqrt g, \sqrt g \rangle_{\nu_{\bp_*}^N} \Big) dt.
\end{equation}

With this aim, we set $\psi := \sqrt g$ and recall the definitions of the symmetric
and anti-symmetric part $L_N^{ex,s}, L_N^{ex,a}$ from \eqref{LNexs} and \eqref{LNexa}.
Since $\langle -L_N^{ex,a} \psi, \psi \rangle_{\nu_{\bp_*}^N}=0$,
we can expand the Dirichlet form $\langle -L_N \psi, \psi \rangle_{\nu_{\bp_*}^N}$ as 
\begin{align*} 
  \langle -L_N \psi, \psi \rangle_{\nu_{\bp_*}^N}
   = N^2 \langle -L_N^{ex,s} \psi, \psi \rangle_{\nu_{\bp_*}^N} + N^2 \langle -L_N^c \psi, \psi \rangle_{\nu_{\bp_*}^N}.
\end{align*}

Next we localize the Dirichlet forms $\langle -L_N^{ex,s} \psi, \psi \rangle_{\nu_{\bp_*}^N}$ and $ \langle -L_N^c \psi, \psi \rangle_{\nu_{\bp_*}^N}$. We use that $\nu_{\bp_*}^N$ is invariant under $L_N^{ex,s}$ to rewrite \comm{[This is a result which holds for all Markov chains; see e.g. KL App 1 Thm 9.2]}
\begin{align*}
  \langle -L_N^{ex,s} \psi, \psi \rangle_{\nu_{\bp_*}^N}
  = \frac12 \sum_{\substack{ x,y \in \T_N^d \\ |y - x| = 1 }} \sum_{v \in \cV} E_{\nu_{\bp_*}^N} \Big[ \eta_x \big( 1- \eta_y \big)
 \big(  \psi \left(  \eta^{x,y,v}\right)  - \psi \left(\eta\right)  \big)^2 \Big]
=: \sum_{\substack{ x,y \in \T_N^d \\ |y - x| = 1 }}  A_{xy},
\end{align*} 
where $A_{xy} \geq 0$ for all $x,y$. Then,
\begin{align*}
  \langle -L_N^{ex,s} \psi, \psi \rangle_{\nu_{\bp_*}^N} 
  &= \sum_{x\in\T_N^d} \sum_{|y - x| = 1} \frac1{|\Lambda_M|} \sum_{z \in \Lambda_{x,M}} A_{xy} \\
  &= \frac1{ |\Lambda_M|} \sum_{z\in\T_N^d}  \sum_{x \in {\Lambda_{z,M}}} \sum_{|y - x| = 1} A_{xy} \\ 
  &\geq \frac1{ |\Lambda_M|} \sum_{z\in\T_N^d} {\bigg(} \sum_{\substack{ x,y \in {\Lambda_{z,M}} \\ |y - x| = 1 }} A_{xy} {\bigg)} \\
  &= \frac1{ |\Lambda_M|} \sum_{z\in\T_N^d} \langle -L_{{\Lambda_{z,M}}}^{ex,s} \psi, \psi \rangle_{\nu_{\bp_*}^N},
\end{align*} 
where 
\[
  L_{{\Lambda_{z,M}}}^{ex,s} \psi (\eta)
  := \sum_{\substack{ x,y \in {\Lambda_{z,M}} \\ |y - x| = 1 }} \sum_{v\in\cV} \eta_x \big( 1- \eta_y \big)
\left\{  \psi \left(  \eta^{x,y,v}\right)  - \psi \left(\eta\right)  \right\}.
\]
In the last step we have used that $\nu_{\bp_*}^N$ is invariant under $L_{{\Lambda_{z,M}}}^{ex,s}$ for any $z$.
The Dirichlet form $\langle-L_N^c\psi, \psi\rangle_{\nu_{\bp_*}^N}$ can be rewritten similarly.
Indeed, since we have
\begin{align*}
  \langle -L_N^{c} \psi, \psi \rangle_{\nu_{\bp_*}^N}
  = \frac12 \sum_{x\in\T_N^d} \sum_{q \in \cQ} E_{\nu_{\bp_*}^N} \Big[ p(x,q,\eta)
 \big(  \psi \left(  \eta^{x,q}\right)  - \psi \left(\eta\right)  \big)^2 \Big] =: \sum_{x\in\T_N^d} B_x,
\end{align*}
it follows that
\begin{align*}
  \langle -L_N^{c} \psi, \psi \rangle_{\nu_{\bp_*}^N}
  &= \sum_{x\in\T_N^d} \frac1{ |\Lambda_M|} \sum_{z\in \Lambda_{x,M}} B_x \\
  &= \frac1{ |\Lambda_M|} \sum_{z\in\T_N^d} \sum_{x\in \Lambda_{z,M}} B_x \\
  &= \frac1{ |\Lambda_M|} \sum_{z\in\T_N^d} \langle -L_{{\Lambda_{z,M}}}^{c} \psi, \psi \rangle_{\nu_{\bp_*}^N},
\end{align*}
where
\[
  L_{{\Lambda_{z,M}}}^{c} \psi (\eta)
  := \sum_{x \in {\Lambda_{z,M}}} \sum_{q \in \cQ} p(x,q,\eta)
\left\{  \psi \left(  \eta^{x,q}\right)  - \psi \left(\eta\right)  \right\}.
\]
In conclusion, Taking $L_{\Lambda_{z,M}}=L_{\Lambda_{z,M}}^{ex,s} + L_{\Lambda_{z,M}}^c$ we get
\[
  \langle -L_N \psi, \psi \rangle_{\nu_{\bp_*}^N}
  \geq \frac{N^2}{|\Lambda_M|} \sum_{z\in\T_N^d} \langle -L_{{\Lambda_{z,M}}} \psi, \psi \rangle_{\nu_{\bp_*}^N}.
\]
Substituting this in \eqref{pfyq} and bringing $\sup_g$ inside $\sum_x$, the expression in \eqref{pfyq} is bounded from above by
\begin{equation} \label{pfyp}
  \frac{N^2}{\gamma |\Lambda_M|} \sum_{x \in \T_N^d} \int_0^T \sup_g \Big( \frac{\gamma |\Lambda_M|}{N^2} \langle a_x m_x, g \rangle_{\nu_{\bp_*}^N} - \langle -L_{{\Lambda_{x,M}}} \sqrt g, \sqrt g \rangle_{\nu_{\bp_*}^N} \Big) dt,
\end{equation} 
where we recall that the supremum is carried over all functions $g:X_N\to[0,\infty)$ satisfying $\int_{X_N} g d\nu_{\bp_*}^N=1$.
 
Next, we bound the integrand in \eqref{pfyp} from above by a version which is localized on $\Lambda_{x,M}$. The $x$-dependence will be of little importance. Therefore, we will not always adopt it explicitly in the notation that follows. In addition, we abuse notation and abbreviate $L_M := L_{{\Lambda_{x,M}}}$.  \comm{[Note: $x$-dependence is not important.]}
 
Let $g_M$ be the conditional expectation $g_M=E_{\nu_{\bp_*}^N}[g|\{\eta_z\}_{z\in\Lambda_{x,M}}]$.  
Since $m_x$ depends only on the variables $\{\eta_z\}_{z\in\Lambda_{x,M}}$ for $M$ large enough,
we have $\langle a_x m_x, g \rangle_{\nu_{\bp_*}^N} = \langle a_x m_x, g_M \rangle_{\nu_{\bp_*}^M}$, where $\nu_{\bp_*}^M$ is the restriction of $\nu_{\bp_*}^N$ to $X_M := (\{0,1\}^\cV)^{\Lambda_{x,M}}$.
On the other hand, since the application $g\mapsto\langle -L_M \sqrt g, \sqrt g \rangle_{\nu_{\bp_*}^N}$ is convex (see \cite[Appendix 1, Corollary 10.3]{KL}), we have \comm{by Jensen}
\begin{align*}
\langle -L_M \sqrt g, \sqrt g \rangle_{\nu_{\bp_*}^N} \ge
\langle -L_M \sqrt {g_M}, \sqrt {g_M} \rangle_{\nu_{\bp_*}^M}.
\end{align*}
Therefore \eqref{pfyp} is bounded from above by
\begin{equation} \label{pfyp-1}
  \frac{N^2}{\gamma |\Lambda_M|} \sum_{x \in \T_N^d} \int_0^T \sup_g \Big( \frac{\gamma |\Lambda_M|}{N^2} \langle a_x m_x, g \rangle_{\nu_{\bp_*}^M} - \langle -L_M \sqrt g, \sqrt g \rangle_{\nu_{\bp_*}^M} \Big) dt,
\end{equation} 
where now the supremum is carried over all functions $g:X_M\to[0,\infty)$ satisfying $\int_{X_M} g d\nu_{\bp_*}^M = 1$. 
  
Let $g$ be such a function.
Recall the definition of $\cD_M$ above \eqref{mic-can space}.
For each $\bi\in\cD_M$, let $X_{M,\bi} = \{\eta\in X_M: \bI_x^M (\eta) = \bi\}$ 
and $\nu_{M,\bi}$ be the canonical measure $\nu_{M,\bi}(\cdot)=\nu_{\bp_*}^M(\cdot|X_{M,\bi})$.
Since $X_M = \cup_{\bi\in\cD_M} X_{M,\bi}$ is a disjoint union, we can write  \comm{[Here we are simply splitting the Markov chain into its irreducible components.]}  
\begin{align*}
\langle a_x m_x, g \rangle_{\nu_{\bp_*}^M}
&= \sum_{\bi\in\cD_M} \int_{X_M} {\bf 1}_{X_{M,\bi}}a_xm_x g d\nu_{\bp_*}^M\\ 
&= \sum_{\bi\in\cD_M} c_{M,\bi}(g)\int_{X_{M,\bi}} a_xm_x g_{M,\bi} d\nu_{M,\bi},
\end{align*}
where 
\begin{align*}
c_{M,\bi}(g) := \int_{X_M}  {\bf 1}_{X_{M,\bi}}g d\nu_{\bp_*}^M,\quad
g_{M,\bi}(\eta) := \frac{\nu_{\bp_*}^M(X_{M,\bi})}{c_{M,\bi}(g)} g(\eta).
\end{align*}
Similarly, the Dirichlet form $\langle -L_M \sqrt {g}, \sqrt {g} \rangle_{\nu_{\bp_*}^M}$
can be rewritten as
\begin{align*} 
\sum_{\bi\in\cD_M} c_{M,\bi}(g)\int_{X_{M,\bi}} \left(-L_M \sqrt{g_{M,\bi}}\right) \sqrt{g_{M,\bi}} d\nu_{M,\bi}.
\end{align*}
Note that $\int_{X_{M,\bi}} g_{M,\bi}d\nu_{M,\bi}=1$ for each $\bi\in\cD_M$.
Therefore \eqref{pfyp-1} is bounded from above by  \comm{["=1" on p.56 below]} 
\begin{equation} \label{pfyo}
  \frac{N^2}{\gamma |\Lambda_M|} \sum_{x \in \T_N^d} \int_0^T \max_{\bi \in \cD_M} \sup_g \Big( \frac{\gamma |\Lambda_M|}{N^2} \langle a_x m_x, g \rangle_{\nu_{M,\bi}} - \langle -L_M \sqrt g, \sqrt g \rangle_{\nu_{M,\bi}} \Big) dt,
\end{equation}
where now the supremum is carried over all functions $g:X_{M,\bi}\to[0,\infty)$ satisfying \\ $\int_{X_{M,\bi}} g d\nu_{M,\bi}=1$. 

Next, we show that the Rayleigh estimate in \cite[Appendix 3, Theorem 1.1]{KL} applies to \eqref{pfyo}. This application requires:
\begin{itemize}  
  \item a spectral gap estimate,
  \item that $\nu_{M,\bi}$ is invariant under $L_M$, 
  \item that $L_M$ is reversible with respect to $\nu_{M,\bi}$,
  \item $E_{\nu_{M,\bi}} m_x = 0$, and
  \item $|a_x m_x| \leq C_1$ uniformly in $t,x,v,\omega,N$. 
\end{itemize} 
Assumption \ref{a:sg} provides a spectral gap estimate with constant $C_0 M^\kappa$.
Since $\nu_{M,\bi}$ is the canonical measure, it is invariant under $L_M$. Reversibility follows from the self-adjointness of $L_M$ with respect to $\nu_{M,\bi}$, which holds by the construction of $L_M$. Noting that $\nu_{M,\bi} = \nu_{\bp_*}^M ( \cdot | \bI_0^M = \bi)$, we have $E_{\nu_{M,\bi}} m_x = 0$ by construction of $m_x$. Finally, the uniform bound on $|a_x m_x|$ holds by the given uniform bound on $a_x$, $|\bI_x| \leq C$ and $\| h \|_{C^1( B(\bp_*, r) )} < \infty$ for some radius $r = r(\cV) > 0$ and $N$ large enough. \comm{["... invariant under $L_M$": see p.57]}

Applying the Rayleigh estimate yields that \eqref{pfyo} is bounded from above by
\begin{equation} \label{BG:pf:1}
  \frac{\gamma |\Lambda_M|}{N^2} \sum_{x \in \T_N^d} \int_0^T \max_{\bi \in \cD_M} \frac{ a_x^2 \langle (-L_M)^{-1} m_x, m_x \rangle_{\nu_{M,\bi}} }{ 1 - 2 C_0 C_1 M^\kappa \gamma |\Lambda_M| / N^2 },
\end{equation}
provided that the denominator is positive. In fact, the dominator is larger than $\frac12$ for $N$ large enough, which follows from $|\Lambda_M| \leq C M^d$ and the upper bound on $\gamma$ imposed in \eqref{pfyj}. For the numerator, we bound \comm{[To PvM: see pdf Q and A with KT p.11]}
\[
  \langle (-L_M)^{-1} m_x, m_x \rangle_{\nu_{M,\bi}}
  \leq C M^\kappa \langle m_x, m_x \rangle_{\nu_{M,\bi}}
  = C' M^\kappa.
\]
In conclusion, for $N$ large enough, \eqref{BG:pf:1} is bounded from above by
$
  C \gamma M^{\kappa + d} N^{d - 2}. 
$

Collecting all estimates above, we obtain
\[
  \E_{\mu^N} \bigg[ \bigg |\int_0^T \sum_{x \in \T_N^d} a_x m_x dt \bigg | \bigg]
      \leq C \Big( \frac1\gamma + \frac{\gamma M^{\kappa + d} }{N^{2-2a}} \Big) N^{d-2a}.
\]      
Minimize this expression over $\gamma$ yields 
\[ \gamma = M^{-\tfrac{\kappa + d}2} N^{1-a}, \]
which by the imposed bounds on $M$ in \eqref{l:L10:6:assn} satisfies the requirement \eqref{pfyj}.
This proves the desired estimate in Lemma \ref{l:BG:mx}. 
\end{proof}

Next, in Lemma \ref{l:BG:box} below, we bound the second term in \eqref{pfyt}. 

\begin{lem} \label{l:BG:box}
For all $N$
\begin{equation*}
      \E_{\mu^N} \left[ \bigg |\int_0^T \sum_{x \in \T_N^d} a_x E_{\nu_{\bp_*}^N}[f_x \mid \bI_x^M] dt \bigg | \right]
      \leq C \int_0^T  H(\mu_t^N | \nu_t^N) dt
  + C' \Big( \frac{N^{2a}}{M^d} + \frac{M^2}{N^2} \Big) N^{d-2a},
\end{equation*}
where the constants $C, C' > 0$ are independent of $N$. 
\end{lem}

\begin{proof}
First, we apply the uniform bound on $a_x$ to estimate
\begin{equation*} 
\E_{\mu^N} \left[ \bigg |\int_0^T \sum_{x \in \T_N^d} a_x E_{\nu_{\bp_*}^N}[f_x \mid \bI_x^M] dt \bigg | \right]
\leq C \E_{\mu^N} \left[ \int_0^T \sum_{x \in \T_N^d} \big| E_{\nu_{\bp_*}^N}[f_x \mid \bI_x^M] \big| dt \right].
\end{equation*}
By Proposition \ref{p:nutN} there exists an $r > 0$ such that $\bLam$ is defined on the ball $B(\bp_*, 2r)$. Setting,
\begin{align*}
\bar \bI_x^M := \bI_x^M - \bp_x^M, \quad 
 \bp_x^M := \dfrac{1}{|\Lambda_M|}\sum_{y \in {\Lambda_{x,M}}} \bp_y,
\end{align*}
we split the right-hand side above as
\begin{multline} \label{pfxk}
  C \E_{\mu^N} \left[ \int_0^T \sum_{x \in \T_N^d} {\bf 1}_{\{|\bar \bI_x^M| \ge r\}} \big| E_{\nu_{\bp_*}^N}[f_x \mid \bI_x^M] \big| dt \right] \\
  + C \E_{\mu^N} \left[ \int_0^T \sum_{x \in \T_N^d} {\bf 1}_{\{|\bar \bI_x^M| < r\}} \big| E_{\nu_{\bp_*}^N}[f_x \mid \bI_x^M] \big| dt \right].
\end{multline}

We estimate both terms in \eqref{pfxk} separately. We start with the first term. Using that under the event $\{|\bar \bI_x^M| \ge r \}$ we have $1\le r^{-2}|\bar \bI_x^M|^2$, we bound
\begin{align} \label{pfxj}
C \E_{\mu^N} \left[ \int_0^T \sum_{x \in \T_N^d} {\bf 1}_{\{|\bar \bI_x^M| \ge r\}} \big| E_{\nu_{\bp_*}^N}[f_x \mid \bI_x^M] \big| dt \right]
\leq C r^{-2} \| f \|_\infty \int_0^T E_{\mu_t^N}\left[ \sum_{x \in \T_N^d} |\bar \bI_x^M|^2 \right] dt.
\end{align}
Applying the entropy inequality with constant $\lambda > 0$, we get
\[
  E_{\mu_t^N} {\Big[} \sum_{x \in \T_N^d} \big| \bar \bI_x^M \big|^2 {\Big]}
  \leq \frac1\lambda H(\mu_t^N | \nu_t^N) 
  +  \frac1\lambda \log E_{\nu_t^N} {\bigg[} \exp \bigg( \lambda \sum_{x \in \T_N^d} \big| \bar \bI_x^M \big|^2 \bigg) {\bigg]}.
\]
For the second term, we apply the usual argument for taking $\sum_x$ in front, which relies on H\"older's inequality and the observation that $\bar \bI_x^M$ and $\bar \bI_y^M$ are independent under $\nu_t^N$ if $| x - y |_\infty \geq 2M+1$. Details can be found in \cite[Lemma 4.2 and Lemma B.4]{JM1}. This yields \comm{[for 2nd ineq: p.89 top]}
\begin{align*}
  &\frac1\lambda \log E_{\nu_t^N} {\bigg[} \exp \bigg( \lambda \sum_{x \in \T_N^d} \big| \bar \bI_x^M \big|^2 \bigg) {\bigg]} \\
  &\leq \frac1{2^d \lambda |\Lambda_M|} \sum_{x \in \T_N^d} \log E_{\nu_t^N} {\Big[} \exp \Big( 2^d \lambda |\Lambda_M| \big| \bar \bI_x^M \big|^2 \Big) {\Big]} \\
  &\leq \frac1{2^d \lambda |\Lambda_M|} \max_{0 \leq k \leq d} \sum_{x \in \T_N^d} \log E_{\nu_t^N} {\Big[} \exp \Big( 2^d (d+1) \lambda |\Lambda_M| \big( \bar I_{x,k}^M \big)^2 \Big) {\Big]}.
\end{align*}
Next we apply the concentration inequality \cite[Lemma 3.6]{FT}. With this aim, we fix $k$ and take
\[
\lambda := \frac1{2^d (d+1) C_{\cV}^2}, \qquad
C_{\cV} := |\cV| + 2 \sum_{v \in \cV} |v|.
\] 
Since under $\nu_t^N$ the random variables $\{ \frac1{C_\cV}\bar I_{y,k} \}_{y \in {\Lambda_{x,M}}}$ are independent, have mean zero and take values in an interval of length less than $1$, the concentration inequality states that 
\[
  E_{\nu_t^N} {\Big[} \exp \Big( 2^d (d+1) \lambda |\Lambda_M| \big( \bar I_{x,k}^M \big)^2 \Big) {\Big]}
  = E_{\nu_t^N} {\bigg[} \exp \bigg( |\Lambda_M|^{-1} \bigg( \sum_{y \in {\Lambda_{x,M}}} \frac1{C_\cV} \bar I_{y,k} \bigg)^2 \bigg) {\bigg]}
  \leq e^2.
\]
In conclusion, we have obtained the following bound on the first term in the right-hand side of \eqref{pfxk}:
\begin{equation} \label{pfxi}
  C \E_{\mu^N} \left[ \int_0^T \sum_{x \in \T_N^d} {\bf 1}_{\{|\bar \bI_x^M| \ge r\}} \big| E_{\nu_{\bp_*}^N}[f_x \mid \bI_x^M] \big| dt \right]
  \leq C' \int_0^T  H(\mu_t^N | \nu_t^N) dt
  + C'' \frac{N^d}{M^d}.
\end{equation}
 
Next we estimate the second term in the right-hand side of \eqref{pfxk}.
By \cite[Proposition 7.1]{BL} (with $M$ large enough with respect to $\| \bphi \|_\infty$, $\| \nabla \bphi \|_\infty$ and the support of $f$, and noting that its proof applies verbatim to any finite set $\cV$) this term is bounded from above by \comm{[Further notes on application Prop 7.1. 1) the $\ell$ there needs to be st $\supp f \subset \Lambda_\ell$, and thus $\ell$ can be chosen indep of $N$. Since $\Lambda_\ell \subset \Lambda_M$ is required, it needs $M$ to be large enough. 2) while the SIPS in BL is different, Prop 7.1 only considers the invariant measures and projections thereof on ensembles; those ARE all the same as in our setting. Hence, this does not make Prop. 7.1 invalid to our case. 3) since Prop. 7.1 only deals with the expec.\ inside the summand, we get an additional factor $N^d$ below due to $\sum_x$, and $C'$ depends in addition on $T, \| f \|_\infty, \supp f$. 4) see p.91 for hints.]} 
\begin{equation}\label{pfys} 
  C \E_{\mu^N} \left[ \int_0^T \sum_{x \in \T_N^d} {\bf 1}_{\{|\bar \bI_x^M| < r\}} \big| \tilde f_x (\bI_x^M) \big| dt \right] 
  + \frac{C'}{|\Lambda_M|} N^d,
\end{equation}
where $C' > 0$ is a constant which only depends on $T, \| f \|_\infty, \supp f, \cV$  and the function $\tilde f_x$ is, similar to $\tilde h$, defined on $B(\bp_*, 2r)$ by $\tilde f_x(\bq) := E_{\nu_{\bq}^N}[f_x]$. Note that, under the event $\{|\bar \bI_x^M| < r\}$, the argument of $\tilde f_x$ satisfies $\bI_x^M \in B(\bp_*, 2r)$ since \comm{[Above, we cannot concisify `$\| f \|_\infty, \supp f$' by $f$, because the latter depends on $N$ through $\bp_x$ while the former does not.]}
\begin{equation*}
  |\bI_x^M - \bp_*| 
  \leq |\bar \bI_x^M| + |\bp_x^M - \bp_x| + |\bp_x - \bp_*|,
\end{equation*}
where $|\bar \bI_x^M| < r$ and
\begin{align} \label{pfxh}
  |\bp_x^M - \bp_x|
  &\leq \max_{y \in {\Lambda_{x,M}} } |\bp_y - \bp_x|
  = N^{-a} \max_{y \in {\Lambda_{x,M}} } |\bphi_y - \bphi_x|
  \leq C M N^{-1-a} \| \nabla \bphi \|_\infty, \\\notag
  |\bp_x - \bp_*|
  &= N^{-a} |\bphi_x| \leq \| \bphi \|_\infty N^{-a}
\end{align}
are smaller than $\frac r2$ for $N$ large enough.

Next we bound the first term in \eqref{pfys}. Recalling the definition of $f_x$ in \eqref{fx} and that $E_{\nu_{\bq}^N} \bar \bI_x = \bq - \bp_x$, we obtain 
\begin{align*}
  \tilde f_x(\bq) 
  &= E_{\nu_{\bq}^N}[ \tau_x h(\eta) -\tilde h(\bp_x) - \nabla \tilde h (\bp_x) \cdot \bar \bI_x ]  \\
  &=  \tilde h(\bq) -\tilde h(\bp_x) - \nabla \tilde h (\bp_x) \cdot ( \bq - \bp_x ).
\end{align*}
Noting that the last two terms are the negative of the first order Taylor approximation of $\tilde h$ at $\bp_x$, we obtain
\[
  \big| \tilde f_x (\bI_x^M) \big| 
  \leq C |\bI_x^M - \bp_x|^2.
\]
With this bound we estimate the first term in \eqref{pfys} from above by
\begin{equation*} 
C \E_{\mu^N} \left[ \int_0^T \sum_{x \in \T_N^d} \big| \bar \bI_x^M \big|^2 dt \right]
  + C \int_0^T \sum_{x \in \T_N^d} \big| \bp_x^M - \bp_x \big|^2 dt.
\end{equation*}
Recalling \eqref{pfxj}, we have already bounded the first term from above by the expression in \eqref{pfxi}. For the second term we use \eqref{pfxh} to obtain 
\[
  C \int_0^T \sum_{x \in \T_N^d} \big| \bp_x^M - \bp_x \big|^2 dt
  \leq C' M^2 N^{d-2-2a}.
\]
By collecting the estimates above the proof of Lemma \ref{l:BG:box} is complete.
\end{proof}

Finally, we use Lemmas \ref{l:BG:mx} and \ref{l:BG:box} to prove Theorem \ref{t:BG}.

\begin{proof}[Proof of Theorem \ref{t:BG}]
Applying Lemmas \ref{l:BG:mx} and \ref{l:BG:box} to \eqref{pfyt} we get
\begin{align} \label{pfxl}
\E_{\mu^N}\left[ \left| \int_0^T \sum_{x\in\T_N^d} a_{t,x}^N f_{t,x} dt \right| \right] 
\leq 
C \int_0^T H(\mu_t^N|\nu_t^N) dt + C' \bigg( \frac{ M^{\tfrac{\kappa + d}2} }{N^{1-a}} + \frac{N^{2a}}{M^d} + \frac{M^2}{N^2} \bigg) N^{d-2a}.
\end{align}
Taking \comm{[The following choice is based on plugging in $M^\alpha$, ignoring $(M/N)^2$, computing the critical values of $\alpha$ for which either of the other two terms equal $1$, and then taking $\alpha$ as the average of these critical values.]}
\[
  M = \big \lfloor N^{\tfrac{\kappa a + d}{d(\kappa + d)}} \big \rfloor
\]
and relying on $a < \frac d{\kappa + 2d}$,
it follows from a direct computation that the required bounds on $M$ in \eqref{l:L10:6:assn} are met and that the factor in parentheses in the right hand side of \eqref{pfxl} is bounded from above by $N^{- \e_0}$ with
\begin{equation} \label{pfyh}
  \e_0 
  := \min \Big\{ 
     \frac12 - a \frac{\kappa + 2d}{2d},  \
     \frac d{\kappa + d} - a \frac{\kappa + 2d}{\kappa + d}, \
     2 \Big( 1 - \frac1{\kappa + d} - a \frac \kappa{\kappa + d} \Big)
     \Big\} > 0.
\end{equation}
This completes the proof.
\end{proof}

\section{Proof of Proposition \ref{p:cV} on admissible velocity sets $\cV$}
\label{s:SG}

This section is the proof of Proposition \ref{p:cV}. We adopt the notation in and leading up to Proposition \ref{p:cV}.

\subsection{Proof of Proposition \ref{p:cV}\ref{p:cV:AV}}
  
Note that $v\in\cV$ implies $\bv = (1, v_* \pm v_\ell) \in \R \times \R^d$ for some $1 \leq \ell \leq n$ and that $\sum_{v \in \cV} v = 2n v_*$. Hence,
\begin{align*}
  \sum_{v \in \cV} \bv \otimes \bv
  = \sum_{v \in \cV} \begin{bmatrix}
  1 & v^T \\ v & v \otimes v
\end{bmatrix}      
  = 2 \begin{bmatrix}
  n & n v_*^T \\ n v_* & n v_* \otimes v_* + \sum_{\ell = 1}^n v_\ell \otimes v_\ell
\end{bmatrix}.
\end{align*}  
Thus, for any $\ba := (a_0, a) \in \R \times \R^d$
\begin{align*}
  \frac12 \ba^T \Big( \sum_{v \in \cV} \bv \otimes \bv \Big) \ba
  = n \big( a_0^2 + 2 a_0 (a \cdot v_*) + (a \cdot v_*)^2 \big) + \sum_{\ell = 1}^n (a \cdot v_\ell)^2.
\end{align*} 
Since $\{v_\ell\}_{\ell=1}^n$ spans $\R^d$, there exists $1 \leq m \leq n$ such that $|a \cdot v_m| = c |a|$ for some $c > 0$ independent of $a$. Then, \comm{[as otherwise $\lrang{ v_\ell }_\ell \in \lrang{ a }^\perp$ of dimension $d-1$, which contradicts]} 
\begin{align*}
  \sum_{\ell = 1}^n (a \cdot v_\ell)^2 
  \geq (a \cdot v_m)^2 
  = c^2 |a|^2.
\end{align*}
Using this together with (set $\delta := c^2 [n (|v_*|^2 + 1)]^{-1} > 0$)
\begin{align*}
  2 a_0 (a \cdot v_*)
  = 2 \frac{a_0}{ \sqrt{1 + \delta} } (\sqrt{1 + \delta} \, a \cdot v_*)
  \geq - a_0^2 \Big( 1 - \frac\delta{1 + \delta } \Big) - (1 + \delta) (a \cdot v_*)^2,
\end{align*}
we obtain
\begin{align*}
  \frac12 \ba^T \Big( \sum_{v \in \cV} \bv \otimes \bv \Big) \ba
  &\geq n \Big( \frac\delta{1 + \delta } a_0^2 - \delta (a \cdot v_*)^2 + \delta (|v_*|^2 + 1) |a|^2 \Big) \\
  &\geq n \delta \Big( \frac{a_0^2}{1 + \delta }  + |a|^2 \Big) 
  \geq \frac{n \delta}{1 + \delta} |\ba|^2.
\end{align*}
Hence, $\sum_{v \in \cV} \bv \otimes \bv$ is positive definite, and thus invertible.

\subsection{Proof of Proposition \ref{p:cV}\ref{p:cV:sg}} 
   
Let $M \geq 1$ and $\bi \in \cD_M$ be given. Recall the micro-canonical space $Y_M(\bi)$ and the corresponding uniform measure $\nu_{M,\bi}$ defined in \eqref{mic-can space} and \eqref{mic-can measure}, respectively.
Any generic constant $C$ that appears below in this section is independent of $(M,\bi)$.

For the proof of Proposition \ref{p:cV}\ref{p:cV:sg} we introduce some notation and establish several lemmas. 
Recall the restricted generators $L_M^{ex,s}$, $L_M^{c}$ introduced in \eqref{restricted operators}.
Define furthermore the mean-field symmetric exclusion generator
\begin{equation*} 
  \tilde L_M^{ex,s}f(\eta) := \frac1{|\Lambda_M|}\sum_{v\in\cV} \sum_{x,y\in\Lambda_M} \eta_x(v)\big(1-\eta_y(v)\big) \big[f(\eta^{x,y,v})-f(\eta)\big],
\end{equation*}
and similarly the mean-field collision operator
\begin{equation*} 
  \tilde L_M^cf(\eta) := \frac1{|\Lambda_M|^3}\sum_{q \in \cQ} \sum_{r\in\Lambda_M^4} p(r,q,\eta)\big[f(\eta^{r,q})-f(\eta)\big],
\end{equation*}
where for $r=(x,y,x',y')$ and $q=(v,w,v',w')$,
\begin{equation*}
  p(r,q,\eta)=\eta_x(v)\eta_y(w)\big[1-\eta_{x'}(v')\big]\big[1-\eta_{y'}(w')\big],
\end{equation*}
and $\eta^{r,q}$ is the configuration obtained from $\eta$ by flipping $\eta_x(v)$, $\eta_y(w)$, $\eta_{x'}(v')$ and $\eta_{y'}(w')$.
Direct calculations show that \comm{[p.77]}
\begin{align} \label{LMexs}
  \langle f,-L_M^{ex,s}f \rangle_{\nu_{M,\bi}} 
  &= \frac14  \sum_{v \in \cV} \sum_{\substack{x,y\in\Lambda_M \\  |x-y|=1}} E_{\nu_{M,\bi}} \left[ \big(f(\eta^{x,y,v})-f(\eta)\big)^2 \right],\\  
  \langle f,-L_M^{c}f \rangle_{\nu_{M,\bi}} 
  &= \frac12 \sum_{q \in Q} \sum_{x \in\Lambda_M} E_{\nu_{M,\bi}} \left[ p(x,q,\eta)\big(f(\eta^{x,q})-f(\eta)\big)^2 \right],\\\label{tLMexs}
  \langle f,-\tilde L_M^{ex,s}f \rangle_{\nu_{M,\bi}} 
  &= \frac1{ 4 |\Lambda_M|} \sum_{v \in \cV} \sum_{x,y\in\Lambda_M } E_{\nu_{M,\bi}} \left[ \big(f(\eta^{x,y,v})-f(\eta)\big)^2 \right],\\ \label{tLMc}
  \langle f,-\tilde L_M^cf \rangle_{\nu_{M,\bi}} 
  &= \frac1{2|\Lambda_M|^3} \sum_{q \in \cQ} \sum_{r\in\Lambda_M^4} E_{\nu_{M,\bi}} \left[ p(r,q,\eta)\big(f(\eta^{r,q})-f(\eta)\big)^2 \right].
\end{align} 

Let $\bK=(K_v)_{v\in\cV}$ be the number of particles for each species:
\begin{equation} \label{bK}
  \bK(\eta)=(K_v(\eta))_{v\in\cV}, \quad K_v(\eta) := \sum_{x\in\Lambda_M} \eta_x(v).
\end{equation}
Since the exclusion dynamics conserves $\bK$, for $\bk\in\{0,\ldots,|\Lambda_M|\}^\cV$ the corresponding micro-canonical ensemble reads
\begin{equation*}
  Y_{M,\bk} := \big\{\eta\in \{0,1\}^{\Lambda_M\times\cV} \,|\, \bK(\eta)=\bk\big\}.
\end{equation*} 
Note that the mass-momentum vector $\bI_0^M(\eta)$ is determined by $\bK(\eta)$:
\begin{equation*} 
  \bI_0^M(\eta) 
  = \frac1{|\Lambda_M|} \sum_{x\in\Lambda_M } \sum_{v \in \cV}  \eta_x(v) \bv
  = \frac1{|\Lambda_M|} \sum_{v\in\cV} K_v(\eta) \bv, \quad \forall\,\eta\in Y_{M} (\bi).
\end{equation*}
Let
\begin{equation*}
  \cD_{M,\bi} := \left\{ \bk \in \{0,\ldots,|\Lambda_M|\}^\cV \,\bigg|\, \sum_{v\in\cV} k_v \bv = \bi \right\}
\end{equation*}
be the set of particle number vectors corresponding to $\bi$.
Note that $Y_M(\bi) = \cup_{\bk\in\cD_{M,\bi}} Y_{M,\bk}$.
We use this to disintegrate the uniform measure $\nu_{M,\bi}$ into $\nu_{M|\bk}:=\nu_{M,\bi}(\,\cdot\,|Y_{M,\bk})$ and
\begin{align*} 
  \bar\nu_{M,\bi}(\bk) := \nu_{M,\bi}(Y_{M,\bk}) = \frac1{Z_{M,\bi}}\prod_{v\in\cV} \binom{|\Lambda_M|}{k_v} 
\end{align*}
for a certain normalization constant $Z_{M,\bi}$.

We state several lemmas which we use in the proof of Proposition \ref{p:cV}\ref{p:cV:sg}.
Recall that the constants $C$ appearing in these lemmas do not depend on $M$ or $\bi$.
The first one is a standard estimate for the Dirichlet form of the mean-field type exclusion.
We postpone the proof to the end of this section.
 
\begin{lem}\label{lem:sg-mf}
There exists $C > 0$ such that for any $f : Y_M(\bi) \to \R$, 
\begin{equation*}
  \big\langle f,-\tilde L_M^{ex,s}f \big\rangle_{\nu_{M,\bi}} \le CM^2\big\langle f, -L_M^{ex,s}f \big\rangle_{\nu_{M,\bi}}.
\end{equation*}
\end{lem}

The next lemma is taken from \cite[Lemma 6.2]{BL}, which allows us to replace the original collision dynamics with the long-range one.
Note that it holds generally for any finite velocity set $\cV$.
We omit the proof.

\begin{lem}\label{lem:sg-lr}
There exists $C > 0$ such that for any $f : Y_M(\bi) \to \R$,
\begin{equation*}
  \big\langle f,-\tilde L_M^{c}f \big\rangle_{\nu_{M,\bi}} \le C\big\langle f,-\tilde L_M^{ex,s}f \big\rangle_{\nu_{M,\bi}} + C\big\langle f,-L_M^cf \big\rangle_{\nu_{M,\bi}}.
\end{equation*}
\end{lem}

The next lemma plays the key role in the proof of the spectral gap.
It turns out to be the only part where the specific choice of $\cV$ is necessary.
To state it, we introduce for $\bk \in \cD_{M,\bi}$ and $q=(v,w,v',w')\in\cQ$ the collision kernel
\begin{equation} \label{pqbk}
  p(q,\bk) := k_vk_w(|\Lambda_M|-k_{v'})(|\Lambda_M|-k_{w'})
\end{equation}
and the particle number vector $\bk^q=(k_v^q)_{v\in\cV}\in\{0,\ldots,|\Lambda_M|\}^\cV$ as
\begin{equation*} 
  k_u^q:=
  \begin{cases}
    k_u-1, &u=v,w,\\
    k_u+1, &u=v',w',\\
    k_u, &u \not= v,w,v',w'
  \end{cases}
  \quad \text{when } p(q,\bk)>0
\end{equation*}
and as $\bk^q := \bk$ when $p(q,\bk)=0$.
Notice that $\bk^q \in \cD_{M,\bi}$.

\begin{lem}\label{lem:sg-key}
Assume that $\cV$ is chosen as in \eqref{eq:cV} and that the condition in Proposition \ref{p:cV} \ref{p:cV:sg} holds.
Then, the following hold:
\begin{enumerate}
\item $|\cD_{M,\bi}| \le [(2M+1)^d+1]^{(n-1)}$;
\item Let $\bk^*\in\cD_{M,\bi}$ be a solution to
\begin{align}\label{eq:sg-key}
  \bar\nu_{M,\bi}(\bk) \ge \bar\nu_{M,\bi}(\bk^q), \quad \forall\,q\in\cQ.
\end{align}
For all $\bk\in\cD_{M,\bi}$, there exists an integer $\ell=\ell(\bk) \geq 0$ and a sequence $\bk_{(j)} \subset \cD_{M,\bi}$ with $j=0$, ..., $\ell$ with the following properties:
\begin{itemize}
   \item $\ell \le|\cD_{M,\bi}|$,
   \item $\bk_{(0)}=\bk$ and $\bk_{(\ell)}=\bk^*$,
   \item for each $j=0,\ldots,\ell-1$ there exists $q_j\in\cQ$ such that $\bk_{(j+1)}=[\bk_{(j)}]^{q_j}$,
   \item for each $j=0,\ldots,\ell-1$ we have $\bar\nu_{M,\bi}(\bk_{(j)})\le\bar\nu_{M,\bi}(\bk_{(j+1)})$.
 \end{itemize} 
In particular, $\bar\nu_{M,\bi}(\bk^*)=\max\{\bar\nu_{M,\bi}(\bk);\bk\in\cD_{M,\bi}\}$.
\end{enumerate}
\end{lem} 

The spectral gap corresponding to the collision kernel $p(q,\bk)$ with respect to $\bar\nu_{M,\bi}$ is a direct corollary of Lemma \ref{lem:sg-key}.

\begin{cor}\label{cor:sg-k} For all $g : \cD_{M,\bi} \to \R$ with $E_{\bar\nu_{M,\bi}} [g]=0$,
\begin{align}\label{eq:sg-k}
  E_{\bar\nu_{M,\bi}} \big[g^2(\bk)\big] \le |\cD_{M,\bi}|^2\sum_{q\in\cQ} E_{\bar\nu_{M,\bi}} \left[ p(q,\bk)\big(g(\bk^q)-g(\bk)\big)^2 \right],
\end{align}
where $p(q,\bk)$ is the collision kernel defined in \eqref{pqbk}.
\end{cor}

\begin{rem}
Let $g(\bk):=E_{\nu_{M,\bi}}[f|\bK(\eta)=\bk]$.
Through direct computation, the sum in the right-hand side of \eqref{eq:sg-k} is equal to $2C|\Lambda_M|^3\langle g(\bK), -L_M^{c}g(\bK) \rangle_{\nu_{M,\bi}}$.
Hence, \eqref{eq:sg-k} indeed gives the spectral gap of the collision operator acting on $\bK$.
\end{rem}

\begin{proof}[Proof of Corollary \ref{cor:sg-k}]
Let $\bk^*\in\cD_{M,\bi}$ be any vector that satisfies \eqref{eq:sg-key}.
For each $\bk\in\cD_{M,\bi}$, let $\bk_{(j)}$ be the corresponding sequence in Lemma \ref{lem:sg-key} of length $\ell + 1$.
As $E_{\bar\nu_{M,\bi}} [g]=0$,
\begin{equation*}
  \begin{aligned}
    E_{\bar\nu_{M,\bi}} \big[g^2(\bk)\big] &\le E_{\bar\nu_{M,\bi}} \left[ \big(g(\bk)-g(\bk^*)\big)^2 \right]\\
    &\le E_{\bar\nu_{M,\bi}} \left[ \ell\sum_{j=0}^{\ell-1} \big(g(\bk_{(j)})-g(\bk_{(j+1)})\big)^2 \right].
  \end{aligned}
\end{equation*}
Since $\ell\le|\cD_{M,\bi}|$ and $\bar\nu_{M,\bi}(\bk) \le \bar\nu_{M,\bi}(\bk_{(j)})$ for all $j$,
\begin{equation*}
  \begin{aligned}
    E_{\bar\nu_{M,\bi}} \big[g^2(\bk)\big] 
    &\le |\cD_{M,\bi}| E_{\bar\nu_{M,\bi}} \left[ \sum_{j=0}^{\ell-1} \big(g(\bk_{(j)})-g(\bk_{(j+1)})\big)^2 \right]\\
    &\le |\cD_{M,\bi}|\sum_{\bk\in\cD_{M,\bi}} \sum_{j=0}^{\ell-1} \bar\nu_{M,\bi}(\bk_{(j)})\big(g(\bk_{(j)})-g(\bk_{(j+1)})\big)^2.
  \end{aligned}
\end{equation*}
Denote by $\bk'\sim\bk$ if $\bk'=\bk^q$ for some $q\in\cQ$ satisfying $p(q,\bk)>0$.
Then, \comm{[Over the 2nd inequality, $\sum_\bk$ turns into $|\cD_{M,\bi}|$, $\ell$ turns into $|\cD_{M,\bi}|$, and $\sum_j$ turns into $\sum_{\bk, \bk'}$.]}
\begin{align*}
  E_{\bar\nu_{M,\bi}} \big[g^2(\bk)\big] 
  \le |\cD_{M,\bi}|^2\sum_{\bk\in\cD_{M,\bi}} \sum_{\bk'\sim\bk} \bar \nu_{M,\bi}(\bk)\big(g(\bk)-g(\bk')\big)^2.
\end{align*}
The result follows, since $p(q,\bk)\ge1$ if $\bk^q\not=\bk$.
\end{proof}

Now, we can state the proof of Proposition \ref{p:cV}\ref{p:cV:sg}.
For simplicity, hereafter we adopt the following short notations for the measures:
\begin{equation*}
  \nu:=\nu_{M,\bi}, \quad \nu_\bk = \nu_{M|\bk}, \quad \bar\nu:=\bar\nu_{M,\bi}.
\end{equation*}

\begin{proof}[Proof of Proposition \ref{p:cV}\ref{p:cV:sg}] \label{p:cV:sg:pf}
Without loss of generality, let $E_{\nu} [f]=0$.
Recall that $\nu_\bk=\nu(\,\cdot\,|Y_{M,\bk})$.
With $F(\bk):=E_{\nu_\bk} [f]$ and $\bK$ defined in \eqref{bK}, we have \comm{[For wlog, see p.77]}
\begin{align} \label{pfxz}
  E_{\nu} \big[f^2\big] = E_{\nu} \big[(f-F(\bK))^2\big] + E_{\bar \nu} \big[F^2\big].
\end{align}

We separately estimate both terms on the right-hand side of \eqref{pfxz}.
For the first term, the uniform spectral gap for mean-field exclusion \cite[Lemma 8.8]{Quastel} yields that
\begin{equation*} 
  \begin{aligned}
    E_{\nu} \big[(f-F(\bK))^2\big] &= \sum_{\bk\in\cD_{M,\bi}} \bar\nu(\bk)E_{\nu_\bk} \big[(f-F(\bk))^2\big]\\
    &\le C\sum_{\bk\in\cD_{M,\bi}} \bar\nu(\bk)\big\langle f,-\tilde L_M^{ex,s}f \big\rangle_{\nu_\bk} = C\big\langle f,-\tilde L_M^{ex,s}f \big\rangle_{\nu}.
  \end{aligned}
\end{equation*} 
For the second term, it suffices to show that
\begin{equation}\label{eq:pfsg2}
  E_{\bar \nu} \big[F^2\big] \le C|\cD_{M,\bi}|^2|\Lambda_M|^3\big\langle f, -\tilde L_M^{c}f \big\rangle_{\nu}.
\end{equation}
Indeed, \eqref{eq:pfsg2} together with Lemma \ref{lem:sg-lr} yields that
\begin{equation*}
  E_{\bar \nu} \big[F^2\big] \le C|\cD_{M,\bi}|^2|\Lambda_M|^3 \left( \big\langle f,-\tilde L_M^{ex,s}f \big\rangle_{\nu}+\big\langle f, -L_M^{c}f \big\rangle_{\nu} \right).
\end{equation*}
Since $|\Lambda_M|=(2M+1)^d$ and $|\cD_{M,\bi}|\le CM^{(n-1)d}$,
\begin{equation*}
  E_{\bar \nu} \big[F^2\big] \le C'M^{(2n+1)d} \left( \big\langle f,-\tilde L_M^{ex,s}f \big\rangle_{\nu}+\big\langle f, -L_M^{c}f \big\rangle_{\nu} \right).
\end{equation*}
The proof is then concluded by Lemma \ref{lem:sg-mf}.

We are left with the proof of \eqref{eq:pfsg2}. From Corollary \ref{cor:sg-k} we get
\begin{equation*}
  E_{\bar \nu} \big[F^2\big] \le |\cD_{M,\bi}|^2\sum_{q\in\cQ} E_{\bar\nu} \left[ p(q,\bk)\big(F(\bk^q)-F(\bk)\big)^2 \right].
\end{equation*}
Then, it suffices to show
\begin{equation}\label{eq:pfsg3}
  \frac1{|\Lambda_M|^3}\sum_{q\in\cQ} E_{\bar\nu} \left[ p(q,\bk)\big(F(\bk^q)-F(\bk)\big)^2 \right] 
  \le 2\big\langle f, -\tilde L_M^{c}f \big\rangle_{\nu}.
\end{equation}
Fix $\bk \in \cD_{M,\bi}$ and $q\in\cQ$ such that $p(q,\bk)>0$.
Notice that for $\eta \in Y_{M,\bk}$ and $r=(x,y,x',y')$, we have that $p(r,q,\eta)=1$ implies $\eta^{r,q} \in Y_{M,\bk^q}$.
Meanwhile, for each $\xi \in Y_{M,\bk^q}$, the number of distinct pairs of $\eta \in Y_{M,\bk}$ and $r$ such that $p(r,q,\eta)=1$, $\xi=\eta^{r,q}$ is
\begin{equation*}
  (|\Lambda_M-k_v^q|)(\Lambda_M-k_w^q)k_{v'}^qk_{w'}^q = \frac{p(q,\bk)\nu_{\bk^q}(\eta^{r,q})}{\nu_\bk(\eta)}.
\end{equation*}
Therefore,
\begin{equation*}
  \begin{aligned}
  F(\bk^q) &= \sum_{\xi \in Y_{M,\bk^q}} f(\xi)\nu_{\bk^q}(\xi)\\
  &=\frac1{p(q,\bk)}\sum_{r\in\Lambda_M^4} \sum_{\eta \in Y_{M,\bk}} p(r,q,\eta)f(\eta^{r,q})\nu_\bk(\eta).
  \end{aligned}
\end{equation*}
With the formula above, the left-hand side of \eqref{eq:pfsg3} is rewritten as
\begin{equation}\label{eq:pfsg4}
  \frac1{|\Lambda_M|^3}\sum_{q\in\cQ} E_{\bar\nu} \left[ \frac1{p(q,\bk)}\bigg(\sum_{r\in\Lambda_M^4} E_{\nu_\bk} \Big[p(r,q,\eta)\big(f(\eta^{r,q})-f(\eta)\big)\Big]\bigg)^2 \right].
\end{equation}
Note that for all $\eta \in Y_{M,\bk}$, $\sum_r p(r,q,\eta) = p(q,\bk)$.
The Cauchy--Schwarz inequality then yields that \eqref{eq:pfsg4} is bounded from above by \comm{[PvM: C-S means Jensen here]}
\begin{equation*}
  \frac1{|\Lambda_M|^3}\sum_{q\in\cQ} E_{\bar\nu} \left[ \sum_{r\in\Lambda_M^4} E_{\nu_\bk} \Big[p(r,q,\eta)\big(f(\eta^{r,q})-f(\eta)\big)^2\Big] \right].
\end{equation*}
In view of \eqref{tLMc}, this is equal to $2\langle f,-\tilde L_M^cf \rangle_\nu$, and thus \eqref{eq:pfsg3} follows.
\end{proof}

We are left with the proofs of Lemmas \ref{lem:sg-mf} and \ref{lem:sg-key}.
Lemma \ref{lem:sg-key} is proved using the idea in \cite[Lemma 6.3]{BL}.

\begin{proof}[Proof of Lemma \ref{lem:sg-key}]
From \eqref{eq:cV}, $\bk\in\cD_{M,\bi}$ means
\begin{equation*}
  \bi = \sum_{\ell=1}^n k_{v_*+v_\ell}\begin{pmatrix} 1 \\ v_*+v_\ell \end{pmatrix} + \sum_{\ell=1}^n k_{v_*-v_\ell}\begin{pmatrix} 1 \\ v_*-v_\ell \end{pmatrix}.
\end{equation*}
For simplicity, we denote $k_{\pm\ell}=k_{v_* \pm v_\ell}$ for $\ell=1$, ..., $n$.
Then,
\begin{equation}\label{pfxw}
  \bi - \begin{pmatrix} 1 \\ v_* \end{pmatrix}\sum_{\ell=1}^n (k_\ell+k_{-\ell}) = \sum_{\ell=1}^n (k_\ell-k_{-\ell})\begin{pmatrix} 0 \\ v_\ell \end{pmatrix}.
\end{equation}
From the $0$th coordinate in \eqref{pfxw} we obtain that $\sum_{\ell=1}^n (k_\ell+k_{-\ell}) = i_0$ is independent of $\bk$. Then, from the other coordinates and the condition imposed on $\{v_\ell\}$, the values of
\begin{equation*}
  \alpha_\ell := k_\ell-k_{-\ell}, \quad \ell=1,\ldots,n
\end{equation*}
are also independent of the choice of $\bk$.
Therefore, $\bk\in\cD_{M,\bi}$ are completely determined by $k_\ell$, $\ell=1$, ..., $n-1$.
Since $0 \le k_\ell \le |\Lambda_M| = (2M+1)^d$, the first part of Lemma \ref{lem:sg-key} is proved.

We turn now to the second part.
From the observation above,
\begin{equation}\label{eq:vbar-simple}
  \bar\nu(\bk) = \frac1{Z_{M,\bi}}\prod_{\ell=1}^n \binom{|\Lambda_M|}{k_\ell}\binom{|\Lambda_M|}{k_\ell - \alpha_\ell}, \quad \forall\,\bk\in\cD_{M,\bi}.
\end{equation}
Let $\cQ(\bk)$ be the collection of all possible collisions in a configuration $\eta \in Y_{M,\bk}$.
In other words,
\begin{equation*}
  \cQ(\bk):=\big\{q\in\cQ;\ p(q,\bk)>0\big\}, \quad \forall\,\bk\in\cD_{M,\bi}.
\end{equation*}
From the condition on $\{v_\ell\}$, a collision can happen only between a pair of particles with velocities $v_* \pm v_\ell$. Moreover, the velocities after such a collision have to be the pair $v_* \pm v_{\ell'}$ for some $\ell'\not=\ell$.
For $\alpha\in\{-|\Lambda_M|,-|\Lambda_M|+1,\ldots,|\Lambda_M|\}$, define
\begin{equation*}
  \begin{aligned}
    h_\alpha(k) &:= \left[ \binom{|\Lambda_M|}{k+1}\binom{|\Lambda_M|}{k-\alpha+1} \right]^{-1} \binom{|\Lambda_M|}{k}\binom{|\Lambda_M|}{k-\alpha}\\
    &= \frac{(k+1)(k-\alpha+1)}{(|\Lambda_M|-k)(|\Lambda_M|-k+\alpha)},
  \end{aligned}
\end{equation*}
for integers $k$ such that all the binomial coefficients that appeared above are well-defined.
In view of \eqref{eq:vbar-simple}, $\bk\in\cD_{M,\bi}$ satisfies \eqref{eq:sg-key} if and only if
\begin{equation*}
  \frac{\bar\nu(\bk^q)}{\bar\nu(\bk)} 
  = \frac{ h_{\alpha_\ell}(k_\ell-1) }{ h_{\alpha_{\ell'}}(k_{\ell'})} \le 1, \quad \forall\,1 \le \ell \not= \ell' \le n.
\end{equation*}
Note that it is also fulfilled when $\ell=\ell'$ since $h_\alpha$ is strictly increasing.
Hence, we obtain the following equivalent statement for \eqref{eq:sg-key}:
\begin{equation}\label{pfxu}
  \max_{\ell=1,\ldots,n} h_{\alpha_\ell}(k_\ell-1) \le \min_{\ell=1,\ldots,n} h_{\alpha_\ell}(k_\ell).
\end{equation}
Notice that there is at least one solution to \eqref{pfxu}, since $\cD_{M,\bi}$ is a finite set and the element $\bk$ satisfying $\bar\nu(\bk)=\max\{\bar\nu(\bk');\bk'\in\cD_{M,\bi}\}$ also satisfies \eqref{pfxu}.

We claim that any two solutions $\bk$, $\bk'$ to \eqref{pfxu} satisfy
\begin{equation}\label{pfxu2}
  \max\big\{|k_\ell-k'_\ell|;\,\ell=1,\ldots,n\big\} \le 1.
\end{equation}
To prove this claim, assume by contrast that $k_\ell \ge k'_\ell + 2$ for some $\ell$.
As $\sum_{\ell = 1}^n k_\ell=\sum_{\ell = 1}^n k'_\ell$, there exists an $\ell'$ such that $k_{\ell'} \le k'_{\ell'}-1$.
Since $h_\alpha$ is strictly increasing and $\bk'$ solves \eqref{pfxu},
\begin{equation*}
  h_{\alpha_{\ell'}}(k_{\ell'}) 
  \le h_{\alpha_{\ell'}}(k'_{\ell'}-1) 
  \le h_{\alpha_\ell}(k'_\ell) 
  < h_{\alpha_\ell}(k_\ell-1),
\end{equation*}
which contradicts with the fact that $\bk$ is a solution to \eqref{pfxu}.

Now we iteratively construct the desired sequence $\bk_{(j)}$ from any $\bk_0\in\cD_{M,\bi}$ to an arbitrarily fixed solution $\bk^*$ to \eqref{pfxu}.
To this end, let $\bk_{(0)} := \bk_0$.
Given $\bk_{(j)}$, if there exists $q\in\cQ$ such that $\bar\nu([\bk_{(j)}]^q)>\bar\nu(\bk_{(j)})$, then let $\bk_{(j+1)} := [\bk_{(j)}]^q$. This iteration produces a sequence with distinct elements, and therefore ends at $\bk = \bk_{(j)}$ for some $j \leq |\cD_{M,\bi}|$. By construction, $\bk$ satisfies \eqref{eq:sg-key}, and thus $\bk$ satisfies \eqref{pfxu} too. If $\bk=\bk^*$, then we put $j = \ell$ and the construction of the sequence is complete.
Otherwise, we construct $\bk_{(j+1)}, \bk_{(j+2)}, \ldots$ in the following manner. By \eqref{pfxu2} we may assume without loss of generality that $k_1^*=k_1-1$, $k_2^*=k_2+1$.
Let $q=(v_*+v_1,v_*-v_1,v_*+v_2,v_*-v_2)$.
If we can show that $\bk^q$ is also a solution to \eqref{pfxu}, then $\bar\nu(\bk)=\bar\nu(\bk^q)$ and we can complete the sequence by repeating this procedure for at most $[\tfrac n2]$ times, thanks to \eqref{pfxu2}.
The total length of the sequence is bounded from above by $|\cD_{M,\bi}|$, since all elements are distinct.

To show that $\bk^q$ solves \eqref{pfxu}, we only need
\begin{equation}\label{pfxu3}
  h_{\alpha_2}(k_2) \le \min_{\ell=1,2} h_{\alpha_\ell}(k_\ell^q), \quad \max_{\ell=1,2} h_{\alpha_\ell}(k_\ell^q-1) \le h_{\alpha_1}(k_1-1).
\end{equation}
Indeed, as $k_\ell^q=k_\ell$ for all $\ell\not=1$, $2$, from \eqref{pfxu3} we obtain that
\begin{equation*}
  \max_{\ell=1,\ldots,n} h_{\alpha_\ell}(k_\ell^q-1) \le \max_{\ell\not=2} h_{\alpha_\ell}(k_\ell-1), \quad \min_{\ell\not=1} h_{\alpha_\ell}(k_\ell) \le \min_{\ell=1,\ldots,n} h_{\alpha_\ell}(k_\ell^q).
\end{equation*}
The desired result then follows, since $\bk$ satisfies \eqref{pfxu}.
To show the first inequality in \eqref{pfxu3}, note that $h_{\alpha_2}(k_2)<h_{\alpha_2}(k_2^q)$ follows from the monotonicity of $h_{\alpha_2}$. Since both $\bk$ and $\bk^*$ satisfy \eqref{pfxu}, we have
\begin{equation} \label{pfxq}
  h_{\alpha_1}(k_1^q) = h_{\alpha_1}(k_1-1) \le h_{\alpha_2}(k_2) = h_{\alpha_2}(k_2^*-1) \le h_{\alpha_1}(k_1^*).
\end{equation}
Since $k_1^q=k_1^*$, all inequalities in \eqref{pfxq} are equalities, and thus $h_{\alpha_2}(k_2)=h_{\alpha_1}(k_1^q)$. This proves the first inequality in \eqref{pfxu3}.
The second one is proved similarly.
\end{proof}

\begin{proof}[Proof of Lemma \ref{lem:sg-mf}]
In view of \eqref{LMexs} and \eqref{tLMexs}, it is sufficient to show that there exists a constant $C > 0$ such that for all $v \in \cV$,
\begin{align} \label{pfxx}
  \sum_{x,y\in\Lambda_M} E_{\nu} \left[ \big(f(\eta^{x,y,v})-f(\eta)\big)^2 \right]
  \leq C M^{d+2} \sum_{\substack{x,y\in\Lambda_M \\ |x-y|=1}} E_{\nu} \left[ \big(f(\eta^{x,y,v})-f(\eta)\big)^2 \right].
\end{align}
Given $x$, $y\in\Lambda_M$ and $\eta \in Y_M(\bi)$, we construct the long-range jump $\eta\to\eta^{x,y,v}$ through a sequence of nearest-neighbor jumps in the following way.
Let $m=(m_1,\ldots,m_d)=y-x \in \Z^d$ and $|m|:=|m_1|+\cdots+|m_d|$.
Define the route $\{x^i;0 \le i \le |m|\}$ by $x^0:=x$,
\begin{equation*}
  x^i:=x^{i-1}+\mathrm{sgn}(m_j)e_j, \quad 
  \text{if} \quad \sum_{j'=1}^{j-1} |m_{j'}|+1 \le i \le \sum_{j'=1}^j |m_{j'}|.
\end{equation*}
Observe that $x^{|m|}=x^0+m=y$.
For any $\eta \in Y_M(\bi)$, define the sequence $\{\eta^i;0 \le i \le 2|m|-1\}$ by $\eta^0:=\eta$,
\begin{equation*}
  \eta^i=
  \begin{cases}
    \big(\eta^{i-1}\big)^{x^{i-1},x^i,v} &\text{if}\ 1 \le i \le |m|,\\
    \big(\eta^{i-1}\big)^{x^{2|m|-i-1},x^{2|m|-i},v} &\text{if}\ |m|+1 \le i \le 2|m|-1.
  \end{cases}
\end{equation*}
A direct computation shows that $\eta^{2|m|-1}=\eta^{x,y,v}$. 
Hence, we obtain a sequence satisfying $\eta^0=\eta$, $\eta^{2|m|-1}=\eta^{x,y,v}$ and for all $1 \le i \le 2|m|-1$ we have $\eta^i=(\eta^{i-1})^{z,z',v}$ for some $z,z' \in \Lambda_M$ with $|z-z'|=1$.

Applying the Cauchy--Schwarz inequality,
\begin{equation*}
  \big(f(\eta^{x,y,v})-f(\eta)\big)^2 \le (2|m|-1)\sum_{i=1}^{2|m|-1} \big(f(\eta^i)-f(\eta^{i-1})\big)^2.
\end{equation*}
Since $\nu$ is invariant under the change of variables $\eta^{i-1}$ to $\eta$ and $|m| \le 2dM$,
\begin{equation*}
  \begin{aligned}
    E_{\nu} \left[ \big(f(\eta^{x,y,v})-f(\eta)\big)^2 \right]
    &\le 4dM\sum_{i=1}^{2|m|-1} E_{\nu} \left[ \big(f(\eta^i)-f(\eta^{i-1})\big)^2 \right]\\
    &\le 8dM\sum_{i=1}^{|m|} E_{\nu} \left[ \big(f(\eta^{x^{i-1},x^i,v})-f(\eta)\big)^2 \right].
  \end{aligned}
\end{equation*}
Summing over all pairs of $x$, $y\in\Lambda_M$, we obtain that the left-hand side of \eqref{pfxx} is bounded from above by
\begin{equation*}
  8dM\sum_{x,y\in\Lambda_M} \sum_{i=1}^{|m|} E_{\nu} \left[ \big(f(\eta^{x^{i-1},x^i,v})-f(\eta)\big)^2 \right].
\end{equation*}
Any bond $(z,z')$ of $\Lambda_M$ appears at most $(2M+1)^{d+1}$ times in the summation above.
To see this, let $z=(z_1,\ldots,z_d)$ and $z'=z+e_j$. A necessary condition for the route connecting $x$ to $y$ to contain $(z,z')$ is
\begin{equation*}
  x_{j'}=z_{j'}, \quad \forall\,j'=j+1,\ldots,d, \quad y_{j'}=z_{j'}, \quad \forall\,j'=1,\ldots,j-1.
\end{equation*}
Therefore, $(z,z')$ appears no more than $(2M+1)^{d+1}$ times.
This yields that the left-hand side of \eqref{pfxx} has an upper bound given by
\begin{equation*}
  8dM(2M+1)^{d+1} \sum_{\substack{ z,z'\in\Lambda_M \\ |z-z'|=1 }} E_{\nu} \left[ \big(f(\eta^{z,z',v})-f(\eta)\big)^2 \right],
\end{equation*}
which completes the proof.
\end{proof}

\newpage

\appendix 

\section{Properties of $\mu_\blam^N$}
\label{a:props:mulamN}

In this section, we prove the properties of $\mu_\blam^N$ stated in Section \ref{s:sett:invmeas}. The computation is standard; we provide it for completeness. We abbreviate the notation for summations as explained in Section \ref{s:ent:prod:not}, and add to it that $q$ always runs over $\cQ$. 

Let $f,g : X_N \to \R$ and $\blam \in \R^{d+1}$ be given. In the computation below up to \eqref{pfxs} we do not use the specific form of $p_N$ in \eqref{pN}. Suppressing the dependence on $v$ in the notation, we write
\begin{align} \label{pfxt}
  \lrang{g, L_N^{ex} f}_{\mu_\blam^N}
  = \sum_{\eta v} \sum_{|x-y|=1} \eta_x (1 - \eta_y) p_N(y-x) \big[ f(\eta^{xy}) - f(\eta) \big] g (\eta) \mu_\blam^N(\eta).
\end{align}
We focus on the part of the summand involving $f(\eta^{xy})$. Changing variables $\zeta = \eta^{xy}$ and noting that $\eta = \zeta^{xy}$, this part of the summand reads as
\begin{align*}
  \sum_{\zeta v} \sum_{|x-y|=1} \zeta_x^{xy} (1 - \zeta_y^{xy}) p_N(y-x) f(\zeta) g (\zeta^{xy}) \mu_\blam^N(\zeta^{xy}).
\end{align*}
Note that $\zeta_x^{xy} = \zeta_y$, $\zeta_y^{xy} = \zeta_x$ and $\mu_\blam^N(\zeta^{xy}) = \mu_\blam^N(\zeta)$. Using these observations and swapping $x$ and $y$ in the summand, the display above reads as
\begin{align*}
  \sum_{\zeta v} \sum_{|x-y|=1} \zeta_x (1 - \zeta_y) p_N(x-y) f(\zeta) g (\zeta^{xy}) \mu_\blam^N(\zeta).
\end{align*}
Replacing $\zeta$ by $\eta$ and substituting this in \eqref{pfxt}, we obtain
\begin{align} \label{pfxs}
  \lrang{g, L_N^{ex} f}_{\mu_\blam^N}
  = \sum_{\eta v} \sum_{|x-y|=1} \eta_x (1 - \eta_y)  \big[ p_N(x-y) g(\eta^{xy}) - p_N(y-x) g(\eta) \big] f (\eta) \mu_\blam^N(\eta).
\end{align}

Note that $L_N^{ex,s}$ and $L_N^{ex,a}$ are obtained from the expression of $L_N^{ex}$ by replacing $p_N$ respectively by its even ($p_N^{even}$) and odd ($p_N^{odd}$) part. Hence, \eqref{pfxs} (which is derived for general $p_N$) yields
\begin{align*}
  \lrang{g, L_N^{ex,s} f}_{\mu_\blam^N}
  &= \sum_{\eta v} \sum_{|x-y|=1} \eta_x (1 - \eta_y) p_N^{even} (y-x) \big[g(\eta^{xy}) -  g(\eta) \big] f (\eta) \mu_\blam^N(\eta) \\
  &= \lrang{L_N^{ex,s} g, f}_{\mu_\blam^N},
\end{align*}
which demonstrates that $L_N^{ex,s}$ is symmetric under $\mu_\blam^N$, and 
\begin{align} \label{pfxp}
  \lrang{g, L_N^{ex,a} f}_{\mu_\blam^N}
  = \sum_{\eta v} \sum_{|x-y|=1} \eta_x (1 - \eta_y) p_N^{odd}(y-x) \big[-g(\eta^{xy}) -  g(\eta) \big] f (\eta) \mu_\blam^N(\eta).
\end{align}
We claim that 
\begin{align} \label{pfxo}
  \sum_{\eta v} \sum_{|x-y|=1} \eta_x (1 - \eta_y) p_N^{odd}(y-x) g(\eta) f (\eta) \mu_\blam^N(\eta) = 0.
\end{align}
Then, it follows from \eqref{pfxp} that $\lrang{g, L_N^{ex,a} f}_{\mu_\blam^N} = - \lrang{L_N^{ex,a} g,  f}_{\mu_\blam^N}$ as desired. The claim \eqref{pfxo} follows from 
\begin{align*}
  \sum_{|x-y|=1} \eta_x (1 - \eta_y) p_N^{odd}(y-x) 
  &= \sum_{|x-y|=1} \frac{\eta_x - \eta_y}2 p_N^{odd}(y-x) \\
  &= \sum_x \sum_{|z| = 1} \frac{\eta_x - \eta_{x+z}}2 p_N^{odd}(z) \\
  &= \frac12 \sum_{|z| = 1} p_N^{odd}(z) \sum_x (\eta_x - \eta_{x+z}) 
  = 0.
\end{align*}

Applying the above to $g=1$ shows that $\mu_\blam^N$ is invariant under $L_N^{ex}$.

Next, we show that $L_N^c$ is symmetric with respect to $\mu_\blam^N$. This implies that $\mu_\blam^N$ is invariant under $L_N^c$.
Suppressing the dependence on $x$ in the notation, we write
\begin{align} \label{pfxr}
  \lrang{g, L_N^c f}_{\mu_\blam^N}
  = \sum_{\eta xq} p(q,\eta) \big[ f(\eta^q) - f(\eta) \big] g (\eta) \mu_\blam^N(\eta).
\end{align}
Similar to the computation leading to \eqref{pfxs}, we focus on the part corresponding to $f(\eta^q)$ and change variables to $\zeta = \eta^q$. Note that $\zeta^{q'}  = \eta$, where for $q = (v,w,v',w')$ we set $q' = (v',w',v,w)$ as the opposite collision. Then, the part corresponding to $f(\eta^q)$ reads as
\begin{align*}
  \sum_{\zeta xq} p(q,\zeta^{q'}) g(\zeta^{q'}) \mu_\blam^N(\zeta^{q'}) f (\zeta).
\end{align*}
Replacing $\zeta$ by $\eta$ and $q'$ by $q$, this reads as
\begin{align*}
  \sum_{\eta xq} p(q',\eta^q) g(\eta^{q}) \mu_\blam^N(\eta^{q}) f (\eta).
\end{align*}
It is easy to check that $p(q',\eta^q) = p(q,\eta)$. Moreover, if $p(q,\eta) \neq 0$, then $\eta_v = \eta_w = 1$, $\eta_{v'} = \eta_{w'} = 0$ and it is easy to see that $\mu_\blam^N(\eta^{q}) = \mu_\blam^N(\eta)$. Substituting these observations in \eqref{pfxr}  we obtain
\begin{align*}
  \lrang{g, L_N^c f}_{\mu_\blam^N}
  = \sum_{\eta xq} p(q,\eta) \big[ g(\eta^{q})  -  g(\eta) \big] f (\eta) \mu_\blam^N(\eta)
  = \lrang{L_N^c g, f}_{\mu_\blam^N}.
\end{align*}

\paragraph{Acknowledgements}

PvM was supported by JSPS KAKENHI Grant Number JP20K14358. 
KT was supported by JSPS KAKENHI Grant Number JP22K13929.
KT and LX are grateful to Professor Shirai for the financial support of LX's visit to Kyushu University.

\end{document}